\documentclass[11pt]{article}
\usepackage{amsmath}
\usepackage{amssymb}
\usepackage{graphicx}
\usepackage{MnSymbol}
\usepackage{amsthm}
\usepackage[dvipsnames]{pstricks}

\newtheorem{theorem}{Theorem}[section]
\newtheorem{lemma}[theorem]{Lemma}
\newtheorem{proposition}[theorem]{Proposition}

{\theoremstyle{definition}}
{\theoremstyle{definition}\newtheorem{definition}[theorem]{Definition}}
{\theoremstyle{definition}}

\numberwithin{equation}{section}

\begin{document}

\def\C{{\mathbb C}}
\def\N{{\mathbb N}}
\def\Z{{\mathbb Z}}
\def\R{{\mathbb R}}
\def\Q{{\mathbb Q}}
\def\K{{\mathbb K}}
\def\Rc{{\mathcal R}}
\def\Qc{{\mathcal Q}}
\def\Kc{{\mathcal K}}
\def\Bc{{\mathcal B}}
\def\F{{\mathcal F}}
\def\H{{\mathcal H}}
\def\E{{\mathcal E}}
\def\B{{\mathcal B}}
\def\rr{{\mathcal R}}
\def\pp{{\mathcal P}}
\def\cF{{\mathcal F}}
\def\cB{{\mathcal B}}
\def\cJ{{\mathcal J}}
\def\cF{{\mathcal F}}
\def\cR{{\mathcal R}}
\def\chH{{\mathcal H}{\mathcal C}{\mathcal H}}
\def\hH{{\mathcal H}{\mathcal H}}

\def\dd{ \{\!\{ }
\def\d{\delta}
\def\c{ \}\!\} }
\def\Hom{\rm Hom}
\def\epsilon{\varepsilon}
\def\kappa{\varkappa}
\def\phi{\varphi}
\def\leq{\leqslant}
\def\geq{\geqslant}
\def\slim{\mathop{\hbox{$\overline{\hbox{\rm lim}}$}}\limits}
\def\ilim{\mathop{\hbox{$\underline{\hbox{\rm lim}}$}}\limits}
\def\supp{\hbox{\tt supp}\,}
\def\dim{{\rm dim}\,}
\def\Ann{{\rm Ann}\,}
\def\Hom{{\rm Hom}\,}
\def\ssub#1#2{#1_{{}_{{\scriptstyle #2}}}}
\def\dimk{{\ssub{\dim}{\K}\,}}
\def\ker{\hbox{\tt ker}\,}
\def\im{\hbox{\tt im}\,}
\def\spann{\hbox{\tt span}\,}
\def\supp{\hbox{\tt supp}\,}
\def\deg{\hbox{\tt \rm deg}\,}
\def\bin#1#2{\left({{#1}\atop {#2}}\right)}
\def\summ{\sum\limits}
\def\maxx{\max\limits}
\def\minn{\min\limits}
\def\limm{\lim\limits}
\def\ootimes{\,{\text{$\scriptstyle\otimes$}}\,}
\def\oo{\otimes}
\def\w{\omega}

\def\K{\mathbb K}
\def\k{\mathbb K}
\def\lll{\langle}
\def\rrr{\rangle}
\def\a{\alpha}
\def\b{\beta}
\def\g{\gamma}
\def\ai{$A_{\infty}$}
\def\l{\langle}
\def\r{\rangle}
\def\lor{k\langle\langle x,y \rangle\rangle}
\def\xd#1{\frac{\partial}{\partial x_{#1}}}
\def\p{\delta}

\def\C{{\mathbb C}}
\def\N{{\mathbb N}}
\def\Z{{\mathbb Z}}
\def\R{{\mathbb R}}
\def\PP{\cal P}
\def\phi{\varphi}
\def\ee{\epsilon}
\def\ll{\lambda}
\def\a{\alpha}
\def\bb{\beta}
\def\D{\Delta}
\def\g{\gamma}
\def\rk{\text{\rm rk}\,}
\def\dim{\text{\rm dim}\,}
\def\ker{\text{\rm ker}\,}
\def\square{\vrule height6pt width6pt depth 0pt}
\def\epsilon{\varepsilon}
\def\phi{\varphi}
\def\kappa{\varkappa}
\def\strl#1{\mathop{\hbox{$\,\leftarrow\,$}}\limits^{#1}}

\vskip1cm

\title{Pre-Calabi-Yau algebras and noncommutative calculus on higher cyclic Hochschild cohomology}

\author{Natalia Iyudu, Maxim Kontsevich}


\date{}

\maketitle

\begin{abstract}

We prove $L_{\infty}$-formality for the higher cyclic Hochschild complex $\chH$ over  free associative algebra or path algebra of a quiver. The $\chH$ complex is introduced as an appropriate tool for the definition of  pre-Calabi-Yau structure. We show that cohomologies of this complex are pure in case of free algebras (path algebras), concentrated in degree zero. It serves as a  main ingredient for the formality proof. For any smooth algebra we choose a small qiso subcomplex in the higher cyclic Hochschild complex, which gives rise to a calculus of highly noncommutative monomials,  we call them $\xi\p$-monomials. The Lie structure on this subcomplex is combinatorially described in terms of $\xi\p$-monomials. This subcomplex and a basis of $\xi\p$-monomials in combination with arguments from Gr\"obner bases theory serves for the cohomology calculations of the higher cyclic Hochschild complex.
The language of  $\xi\delta$-monomials
in particular allows an  interpretation of pre-Calabi-Yau structure
as a noncommutative Poisson structure.


\end{abstract}



\section{Introduction}\label{intro}
The calculations we perform on the higher cyclic Hoshschild complex which lead to the formality proof as well as the construction of this complex itself are inspired by the
notion of pre-Calabi-Yau algebra, which was introduced by
Kontsevich, Vlassopoulos  \cite{KV} (see also the talk \cite{Ktalk}),  Seidel \cite{Seidel}, and Tradler, Zeinalian \cite{TZ}.
It turned out that this structure is present in many different areas, including topology of compact manifolds with boundary, algebraic geometry, symplectic geometry. For example,  Fano varieties are endowed with a pre-Calabi-Yau structure, open Calabi-Yau manifolds have this structure, from the HMS conjecture it is expected that the Fukaya wrapped category of an open symplectic manifold endowed with a pre-Calabi-Yau structure.
Pre-Calabi-Yau structure produces a generalization (for arbitrary genus) of the Tamarkin-Tsygan calculus.
Big class of examples comes from the notion of algebra (category) of finite type, introduced by To\"en and Vaqui\'e \cite{TV}. For these algebras $d$-pre-Calabi-Yau structure on $A$ produces a $(2-d)$-shifted derived Poisson structure on the moduli stack of finite-dimensional $A$-modules.

The philosophy presented in the  Kontsevich and Soibelman paper \cite{KS} says that it is natural to construct a formal noncommutative geometry where the role of manifold played by $A_{\infty}$-algebra, and the role of Poisson structure,  by pre-Calabi-Yau algebra. There were  numerous attempts in the past to introduce a notion of noncommutative Poisson structure. First natural thought in this direction is to try and define it exactly like in commutative case, as a Lie bracket $\{.,.\}:A \times A \to A$ which satisfies the Leibnitz rule $\{a,bc\} =\{a,b\}c+b\{a,c\}$. However
this definition would give a very restricted class of objects. As it was shown by Farkas and Letzter \cite{FL} the only such  Poisson bracket on noncommutative prime ring  is the bracket $[a,b]=ab-ba.$ So this definition turned out to be  not what one would hope for. Another notion of noncommutative Poisson bracket was suggested by Xu \cite{X} and by Block and Getzler \cite{BG}. It had the property that if $A$ endowed with  this noncommutative Poisson structure, then on the centre of $A$, ${\cal Z}(A)$ there is an induced commutative Poisson structure, but the definition did not ensure that there will be an induced Poisson structure on  representation spaces of $A$ or on their moduli. It was desirable to find a notion of noncommutative Poisson structure, which would behave well on the testing ground of representation spaces. This was  achieved in the definition of double Poisson bracket by Van den Bergh \cite{VdB}. The bracket was 'thickened', that is it was defined as a map $\dd.,.\c:A \otimes A \to A \otimes A$ , satisfying  axioms which are certain generalisation of the Leibnitz and the Jacobi identities on $A \oo A$.  In the work of Crawley-Boevey \cite{CB} the definition of noncommutative ($H_0$) Poisson structure was given as
a Lie bracket on zero Hochschild homology $H_0(A)=A/[A,A]$, which can be lifted to a derivation on $A$.
We put these latter developments into the context of pre-Calabi-Yau structures (or rather associated cohomology) and embed noncommutaive bivector fields into a calculus of 'highly noncommutative' words. This gives a new perspective, kind of panoramic aerial view on what is going on, how noncommutative notions involving 'thickening' or factorizations arise.
The attempts to reach a good behaviour on representation spaces was based on the philosophy introduced by Kontsevich and Rosenberg \cite{Kf,KR}, we follow the ideas of this paper throughout the text and pursue some  aspects of those, as well as of their reincarnations, for example, in Kontsevich and Soibelman  \cite{KS}.

We give several equivalent definitions of pre-CY structures, one of them in terms of higher cyclic Hochshild complex. Amongst the advantages of this definition is that it works not only for finite dimensional algebras, or algebras with finite dimensional graded components. We show in section~\ref{dP} the reason why this definition contains as a particular case  the double Poisson bracket.
Roughly speaking, a pre-Calabi-Yau structure is a solution of the the Maurer-Cartan equation with respect to generalized necklace bracket in the higher cyclic Hochschild complex (for precise definition see section~\ref{hH}).

We study the higher cyclic Hochschild complex ($\chH$), its homologies and Lie structure.
 One of the tools we use for the combinatorial description of Lie structure, and for the proof of purity of this complex is  a calculus of noncommutative cyclic words with labels, which we introduce.
 We start with the
 free associative algebra  $A=\k \l X \r$ with finite number of generators $X=\{x_1,...,x_r\}$ and 'labels' $\d_1,...,\d_r$, $\xi$.
 Monomials from free algebra $\k \l X \r$ are written cyclically on the circle and separated by labels.
 These generalised cyclic monomials with labels, which we call $\xi\d$-monomials, represent operations on tuples of monomials from $A$ and form a convenient basis in the small subcomplex of the  higher cyclic Hochschild complex. The $\xi\d$-monomials can be considered  as highly noncommutative words, which can be multiplied not only from the right or from the left, but from $r$ sides, where $r$ is a number of $\p$th in the monomial. The result on formality of the higher cyclic Hochschild complex we got provides $L_{\infty}$ quasi-isomorphism between higher cyclic Hochschild complex and its cohomology. Thus the consideration of homology of the complex instead of the complex itself for the notion of noncommutative Poisson structures
 and of noncommutative polyvector fields becomes justified.

Whenever we are working with the higher cyclic Hochschild complex itself, without embracing in further word combinatorics, we can speak of an arbitrary  smooth algebra $A$,
 as defined in \cite{KR}, i.e. a finitely generated algebra $A$, with  kernel of multiplication being a projective $A$-bimodule.



To deal with the higher cyclic Hochschild complex $\chH(A)=C^{(\bullet)}(A)$
 we choose a small subcomplex $\xi^{(\bullet)}$, quasi-isomorphic to
$C^{(\bullet)}(A).$
 We specify   a particular embedding of the subcomplex $\zeta^{(\bullet)}$ into $C^{(\bullet)}(A)$ (section~\ref{sH}) by choosing a basis  of $\xi\d$-monomials in $\xi^{(\bullet)}$ and
 describing an element of $C^{(\bullet)}$ (which we call an  'operation')  corresponding to a
$\xi\d$-monomial.
 The element of $\prod\limits_{r_1,...,r_n\geq 0} \Hom (\mathop{\oo}\limits_{i=1}^n A^{\oo r_i}, A^{\oo n})$, corresponding to $\xi\d$-monomial  with $m=\sum r_i$ occurrences of labels
$\delta_i, i=\bar{1,r}$, and $n-m$ occurrences of label $\xi$
 is schematically shown in the following picture.

\vskip1cm

\phantom0\hskip4cm
\includegraphics[scale=0.06]{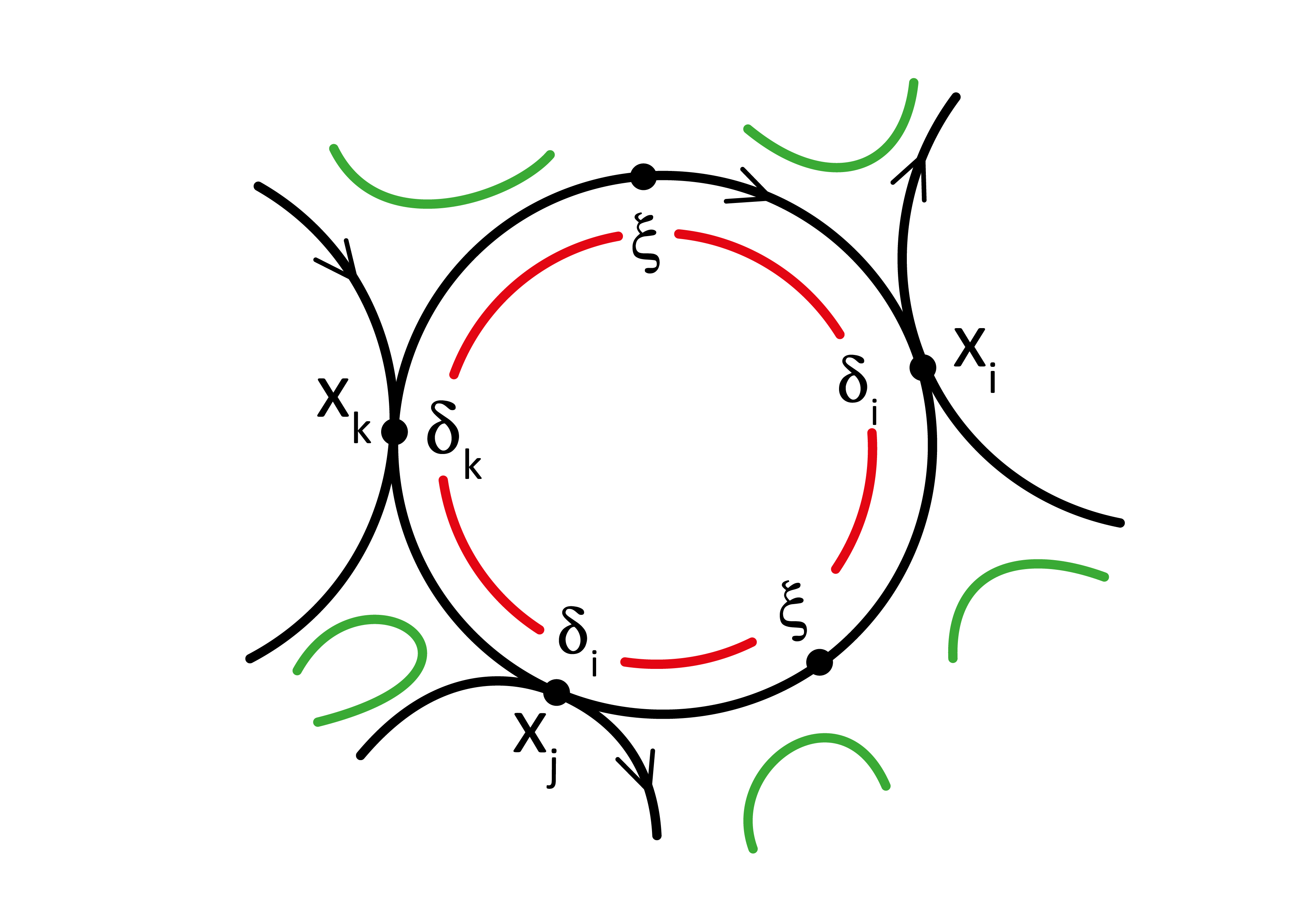}
\hfill

\phantom0

\vskip1cm

The black arches are input monomials from $A$ and green arches are output monomials, consisting of parts of inputs and parts of $\xi\d$-monomial defining the operation. We suppose orientation is clockwise everywhere. One can read the output monomials following this orientation.
In this particular picture we see the $\xi\d$-monomial
which encodes an operation $\Phi: A^{\oo 3}\to  A^{\oo 5}$ (more precisely, $\Phi: A\oo A^0 \oo A \oo A^0 \oo A \to A^{\oo 5})$.

The particular case of  operation encoded by $\xi\d$-monomial with two $\delta$th and no $\xi$th correspond to the Poisson double bracket. Indeed,  the operation $A\otimes A \to A \oo A$,  obtained from such a $\xi\d$-monomial automatically satisfies the double Leibnitz rule.  The double Jacobi identity comes from the Maurer-Cartan equation on the elements of small subcomplex $\zeta^{(\bullet)}$ of the higher Hochschild complex.
This will be explained more precisely in  section \ref{dP}. We show there how the double Poisson bracket invented by Van den Bergh \cite{VdB} as a structure which induce a  Poisson bracket on representation space of algebra,  appear as a particular  pre-Calabi-Yau structure. In \cite{IK} we gave a detailed proof of the following fact.
Any pre-Calabi-Yau structure with $m_4=0$ on arbitrary associative algebra gives rise to a double Poisson bracket according to the formula \cite{IK}:
$$
(*)\quad\quad\quad  \langle g\otimes f,\dd b,a\c\rangle:=\langle m_3(a,f,b), g \rangle,
$$
Moreover, an arbitrary double Poisson bracket can be obtained  from pre-Calabi-Yau structure of special type, with only second and third multiplications $m_2$ and $m_3$ present.
We comment here on the main idea behind this earlier direct proof from the point of view of the definition of pre-Calabi-Yau structure via higher cyclic Hochschild complex.

Coming back to the general situation, we describe the generalised necklace bracket which endows the higher cyclic Hochschild complex with a graded  Lie algebra structure.  In section~\ref{Lie} we show how this bracket works  in terms of $\xi\d$-monomials. By this we not only prove that the small subcomplex  $\zeta^{(\bullet)}_{A}$
is a Lie subalgebra in ${\bf g}=(C^{(\bullet)}_{A}(A), [,]_{g.n})$, but also give a concrete combinatorial formula for this bracket  on $\xi\d$-monomials.
We prove that the bracket $[A,B]$ of two $\xi\d$-monomials $A, B \in \zeta^{(\bullet)}_{A}$ is
$[A,B]=A \circ B - B \circ A$, where $A \circ B$ is a linear combination of  $\xi\d$-monomials obtained from $\xi\d$-monomials  $A$ and $B$ by all possible gluings  of $\delta_j$ of $A$ and $x_j$ of $B$, as shown in the following picture.

\vskip1cm

\phantom0\hskip4cm
\includegraphics[scale=0.2]{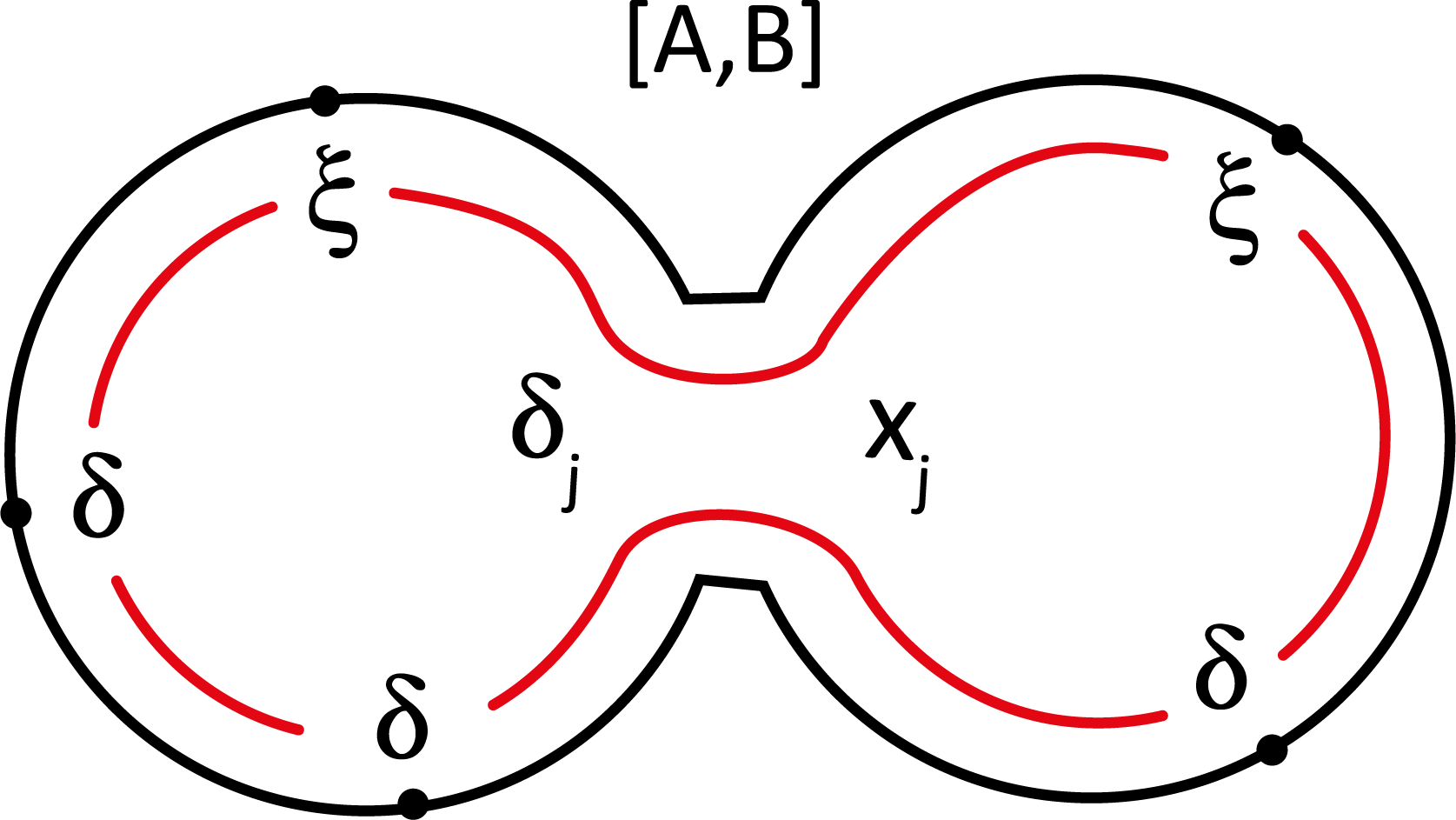}
\hfill

\phantom0

\vskip1cm

Namely, we glue all  $\d_j$ from $\xi\d$-monomial $A$ to a corresponding $x_j$ from $B$, then cut at the place of gluing, and open up to obtain one new $\xi\d$-monomial ($x_j$ and $\d_j$ disappear).

The choice of the basis of $\xi\delta$-monomials in the subcomplex $\xi^{(\bullet)}$ of the  higher cyclic Hochschild complex allows, among other things,  an easy interpretation of pre-Calabi-Yau structure as a noncommutative Poisson structure.
Namely, the $\xi\d$-monomial produce an obvious formal analogue of polyvector field, which in turn create a Poisson structure on the representation space of $A$, via kind of Schouten bracket.

Let us note also the following consequence of our result. In view of quite obvious connection between antisymmetric solutions of Yang-Baxter equation and double Poisson brackets, which was noticed first in \cite{TS}, this our description of all double Poisson brackets
in terms of $\xi\delta$-words,
 provides at the same time a description  of all antisymmetric solutions of the Yang-Baxter equation on the vector space $X$.

In section~\ref{hp}
we concentrate on homological properties of the higher cyclic Hochschild complex and prove its homological purity.
 We again use the small quasi-isomorphic subcomplex $\zeta^{(\bullet)}_{A}$
introduced  in section~\ref{sH}.
From the expression of the differential in the whole dualised bar complex $C^{(\bullet)}_{A}(A)$, which we spell out in section~\ref{sH},  we get a differential in $\zeta=\zeta^{(\bullet)}_{\k}$.

 Note that while the elements of higher cyclic  Hochschild complex are defined as elements of
   ${\rm Hom}(\mathop{\otimes}\limits_{i=1}^N A^{\oo r_i}, A^{\oo N})^{\Z_N}$,
  invariant under $\Z_N$-acton, our homology calculations are reduced to a related non-$\Z_N$-invariant complex  $\tilde\zeta$, corresponding to operations with fixed point. This is possible since the differential commutes with the cyclisation procedure (see Lemma~\ref{unbar}).

 The complex with the fixed point  $\tilde\zeta=\oplus \tilde\zeta^{n}_N$, where
$ \tilde\zeta^{n}=\{$ monomials  $u\in \k\l\xi, x_i, \p_i \r$, starting from $\xi$ or $\p_i$, such that
$\deg_{\p} u=n, \deg_{\p, \xi} u=N\}_{\k}$
has natural   bigrading
 by $\p$-degree, and  by degree with respect to  $\xi$ and  $\p_i$th, $i=\bar{1,r}$, the latter we  call {\it weight}.
 Essential for our considerations is the {\it cohomological grading} by $\xi$-degree: $\zeta=\oplus\zeta(l),$ where $\zeta(l)=\mathop{\oplus}\limits_{N-n=l} \zeta^n_N$.

\begin{theorem}\label{M} Let $A$ be a free algebra with at least two generators,
$A=\K\l x_1,..., x_r \r, r\geq 2$, or a path algebra of a quiver with at least two vertices, $A=P\Qc, |\Qc_0| \geq 2$.
Then the homology of the complex $\tilde\zeta(A)=\oplus \tilde\zeta^n_N$ is sitting in the diagonal $n=N$.
 Consequently, the complex $\tilde\zeta=\oplus \tilde\zeta(l)$, $\, \tilde\zeta(l)=\mathop{\oplus}\limits_{N-n=l} \tilde\zeta_N^n$ is pure, that is its homology is sitting only in the last place of the complex $\tilde\zeta$ with respect to cohomological grading by $\xi$-degree. Homological purity hence holds for the higher cyclic Hochschild complex $C^{(\bullet)}$.
\end{theorem}

In other words, Theorem\ref{M} holds for  path algebra $P\Qc$ of an arbitrary quiver $\Qc$, except for the quiver with one vertex and one loop.

 This purity result is obtained via use of $\xi\delta-monomials$ as a basis of qiso subcomplex and
of the Gr\"obner bases theory in the ideals of free algebra generated by the element defined by the differential of this complex. Similar techniques and the Gr\"obner bases theory in the ideals of path algebras are used for the case when $A$ is the path algebra of a quiver $A=P\Qc.$

As a consequence of purity result we are able to deduce $L_{\infty}$-formality for this complex over  $A=\K\l x_1,..., x_r \r, r\geq 2$ or  $A=P\Qc, |\Qc_0| \geq 2$, using standard arguments related to pertrubation theory, similar to the ones appeared in \cite{D}.

\begin{definition}
The DGLA $(C,d)$ is called {\it formal} if it is quasi-isomorphic to its cohomologies
$H^{\bullet}C.$
\end{definition}

\begin{definition}
The complex $(C,d)$ is $L_\infty-formal$    if it is $L_\infty$ quasi-isomorphic to its cohomologies
$(H^{\bullet}C, 0)$, considered with zero differential, that is there exists an $L_\infty$-morphism, which is a quasi-isomorphism of complexes.
\end{definition}

Since there are more $L_\infty$-morphisms between given DGLAs, than just DGLA morphisms, the notion of $L_\infty$-formality is weaker than formality. But it is exactly what is needed for deformation theory. One of the main points in
 \cite{KP} emphasise that what really determines the deformation functor,  is not just qiso class of a DGLA, but its qiso class as $L_{\infty}$-algebra.  Thus the best thing one can achieve in understanding the deformation theory is to prove
 $L_{\infty}$-formality of corresponding DGLA.


\begin{theorem}  Let $A$ be a free algebra with at least two generators,
$A=\K\l x_1,..., x_r \r, r\geq 2$, or a path algebra of a quiver with at least two vertices, $A=P\Qc, |\Qc_0| \geq 2$.
Then the higher cyclic Hochschild complex $\chH(A)=C^{(\bullet)}(A)=\prod\limits_N C_{cycl}^{(N)} (A)$
 is $L_{\infty}$-formal.
\end{theorem}

Thus the line of our study related to the deformation theory got its best possible outcome in case of  free algebras. The $L_{\infty}$-formality holds also
for such smooth (quasi-free) algebras, as path algebras of quivers, but certainly not  for an arbitrary associative algebras concentrated in degree zero, or for $\Z$-graded free algebras.

\section{Definitions}\label{def}

The typical example of an  algebra in this paper is a free associative algebra $A=\l x_1,...,x_r\r$, the most noncommutative algebra possible. We develop elements of noncommutative geometry based on this algebra following the spirit of \cite{KR, Kf}. For example, we adopt the ideology introduced and developed  in these papers, which says that noncommutative structure should manifest as a corresponding commutative structure on representation spaces. We develop the ideas of these papers further and introduce a calculus of 'highly noncommutative monomials': monomials which can be multiplied not only from the right and from the left, but from any number of specified directions, we call them $\xi\delta$-monomials.

Throughout the text $A$  will be an associative unital algebra over the field $\K$ of characteristic zero, if not specified otherwise.
Denote by $A-mod-A$ the category of all $A$-bimodules, which is the same as $A^e$-modules, i.e. modules over the enveloping $A^e=A\oo A^{op}$. We consider mainly Homs of $A$-bimodules or $A^{\oo N}$-bimodules which we denote $\Hom_{A-mod-A}$ or
$\Hom_{A^{\oo N}-mod-A^{\oo N}}$ respectively.

To give a definition of pre-Calabi-Yau structure as it was originally defined in \cite{Ktalk}, \cite{KV}, \cite{Seidel} we
start with reminding the definition of $A_{\infty}$-algebra, or {\it strong homotopy associative algebra} introduced by Stasheff \cite{St}.

First note, that there are two accepted conventions on grading of an $A_\infty$-algebra. They differ by a shift in numeration of graded components.
In one convention, we call it {\it shifted convention}, each operation has degree 1. While the other, which we call a {\it naive convention} is determined by making the binary operation to have degree 0, hence the degrees
of operations $m_n$ of arity $n$ become $2-n$.
 If the degree of element $x$ in {\it naive convention} is  ${\rm deg}x=|x|$, then shifted degree in $A^{sh}=A[1]$, which fall into {\it shifted convention}, will be ${\rm deg^{sh}}x=|x|'$, where  $|x|'=|x|-1$, since $x \in A^i=A[1]^{i+1}$.

  The formulae for the graded Lie bracket, Maurer-Cartan equations and cyclic invariance of the inner form are somewhat different in different conventions. We mainly  use the {shifted convention}, but sometimes need  the { naive convention} as well.

Let $A$ be a $\Z$ graded vector space $A=\mathop{\oplus}\limits_{n \in \Z} A_n$, and  $C^l(A,A)$ be Hochschild cochains $C^(A,A)=\underline{\Hom}_{}(A^{\oo l}, A)$, for $l\geq 0$, $C^{\bullet}(A,A)= \prod\limits_{k\geq 1} C^l(A,A).$

On
$C^{\bullet}(A,A)[1]$ there is a natural structure of graded pre-Lie algebra, defined via composition:
$$\circ:
C^{l_1}(A,A) \oo C^{l_2}(A,A) \to C^{l_1+l_2-1}(A,A):$$
$$ f \circ g (a_1 \oo ... \oo a_{l_1+l_2-1}) =$$
 $$\sum (-1)^{| g| \sum\limits_{j=1}^{i-1} |a_j|} f(a_1 \oo ... \oo a_{i-1} \oo g(a_i \oo ... \oo a_{i+l_2+1}) \oo ... \oo a_{l_1+l_2-1})$$

 The operation $\circ$ defined in this way does satisfy the graded right-symmetric identity:
$$(f,g,h)=(-1)^{|g||h|}(f,h,g)$$
where
 $$(f,g,h)=(f\circ g)\circ h - f\circ (g \circ h).$$

 As it was shown in  \cite{G} the graded commutator on a graded pre-Lie algebra defines a graded Lie algebra structure.

 Thus the Gerstenhaber bracket $[-,-]_G$:
$$[f,g]_G= f\circ g - (-1)^{| f|| g|} g \circ f$$
makes $C^{\bullet}(A)$ into a graded Lie algebra. Equipped with the derivation $d={\rm ad} \,\, m_2$, $\,\,\, (C^{\bullet}(A), m_2)$ becomes a DGLA, which is a Hochschild cochain complex.

 Graphically the corresponding composition can be depicted as follows.

 \vskip1cm

\phantom0\hskip1cm
\includegraphics[scale=0.9]{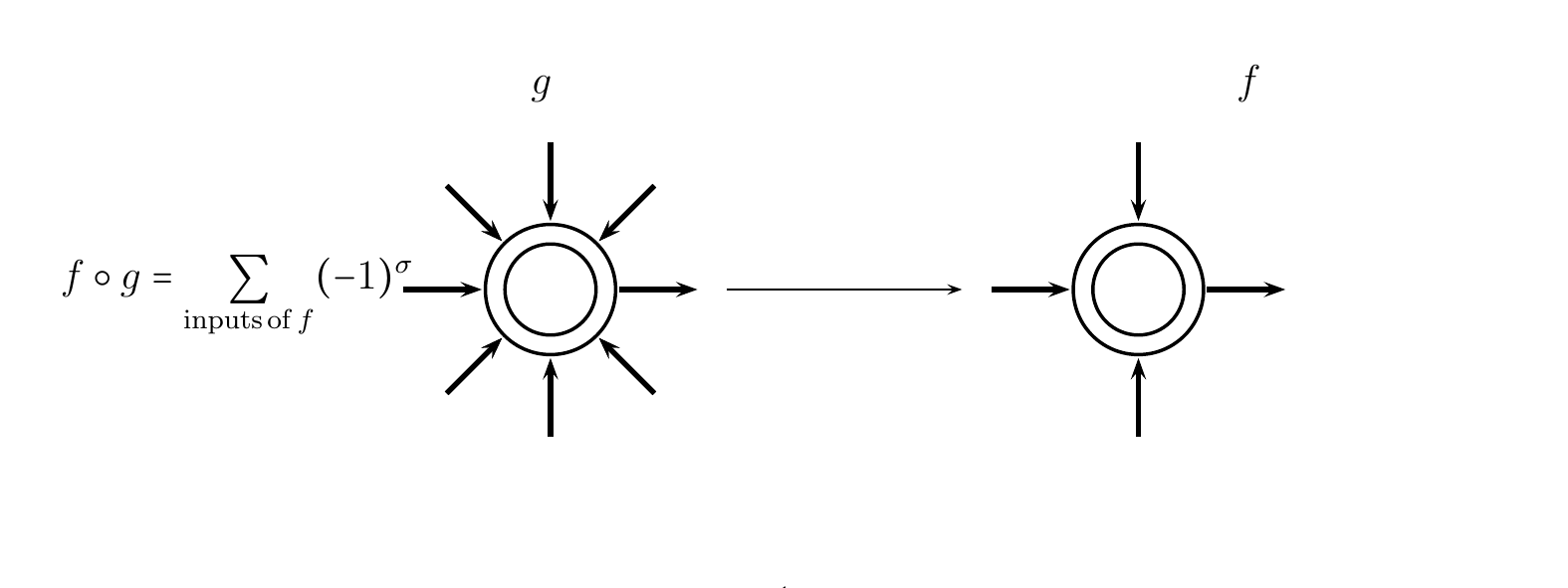}
\hfill
\phantom0

 With respect to the Gerstenhaber bracket $[-,-]_G$  we have  the Maurer-Cartan equation

\begin{equation}\label{MC}
[m^{(1)},m^{(1)}]_G=\sum\limits_{p+q=k+1} \sum\limits_{i=1}^{p-1} (-1)^\epsilon m_p(x_1,\dots,x_{i-1},m_q(x_j,\dots,x_{i+q-1}),\dots,x_k)=0,
\end{equation}
where
$$
\epsilon=|x_1|'+{\dots}+|x_{i-1}|',\qquad |x_i|'=|x_i|-1={\rm deg}x_i-1
$$

The  Maurer-Cartan equation in the { naive convention} looks like:

\begin{equation}\label{MC0}
[m^{(1)},m^{(1)}]=\sum\limits_{p+q=k+1} \sum\limits_{i=1}^{p-1} (-1)^\epsilon m_p(x_1,\dots,x_{i-1},m_q(x_j,\dots,x_{i+q-1}),\dots,x_k)=0,
\end{equation}
where
$$
\epsilon=i(q+1)+q(|x_1|+{\dots}+|x_{i-1}|),
$$

\begin{definition} An element $m^{(1)} \in C^{\bullet}(A,A)[1]$ which satisfies the Maurer-Cartan equation $[m^{(1)}, m^{(1)}]_G=0$ with respect to the Gerstenhaber bracket $[-,-]_G$ is called an
$A_{\infty}$-structure on $A$.

\end{definition}

Equivalently, it can be formulated in a more compact way as a coderivation on the coalgebra of the bar complex of $A$.

In particular, for example, associative algebra with zero derivation $(A, m=m_2^{(1)})$ is an  $A_{\infty}$-algebra. The  component of the Maurer-Cartan equation of arity 3 says that the binary operation from this structure, the multiplication $m_2$ is associative:
$$(ab)c-a(bc)=dm_3(a,b,c)+(-1)^{\sigma}m_3(da,b,c)+(-1)^{\sigma}m_3(a,db,c)+(-1)^{\sigma}m_3(a,b,dc)$$

We can give now one of  definitions of pre-Calabi-Yau structure (in the { shifted convention}).

\begin{definition}\label{pcy1} A d-pre-Calabi-Yau structure on a finite dimensional $A_{\infty}$-algebra $A$ is

 (I). an $A_{\infty}$-structure on $A \oplus A^*[1-d]$,

 (II). cyclic invariant with respect to natural non-degenerate pairing on $A \oplus A^*[1-d]$,
meaning:

$$\l m_n(\a_1,...,\a_n), \a_{n+1}\r=(-1)^{|\a_1|'(|\a_2|'+...+|\a_{n+1}|')}\l m_n(\a_2,...\a_{n+1}),\a_1)\r$$
where the inner form $\l,\r$ on $A\oplus A^*$   is defined naturally as
$\l (a,f),(b,g)\r=f(b)+(-1)^{|g|' |a|'} g(a)$ for $a,b \in A, f,g \in A^*$

(III)
and such that $A$ is an $A_{\infty}$-subalgebra in $A \oplus A^*[1-d]$.
\end{definition}


The signs in this definition, written in shifted convention, are
assigned according to the Koszul rule.
Note, by the way, that in the naive convention,
the
cyclic invariance condition with respect to the natural non-degenerate pairing on $A \oplus A^*[1-d]$ from (II) sounds:
$$\l m_n(\a_1,...,\a_n), \a_{n+1}\r=(-1)^{n+|\a_1|'(|\a_2|'+...+|\a_{n+1}|')}\l m_n(\a_2,...\a_{n+1}),\a_1\r.$$
The appearance of the arity $n$, which influence the sign in this formula, does not really follow the Koszul rule, this is the feature of the { naive convention}, and this is why the  shifted one is preferable.

The cyclic invariance and inner form symmetricity in the {shifted convention} look like:

\begin{equation}\label{cyc1}
\l m_n(\a_1,...,\a_n), \a_{n+1}\r=(-1)^{|\a_1|'(|\a_2|'+...+|\a_{n+1}|')}\l m_n(\a_2,...\a_{n+1}),\a_1)\r.
\end{equation}

\begin{equation}\label{sym2}
\langle x,y \rangle = - (-1)^{|x|'\,|y|'} \langle y,x \rangle
\end{equation}

Since the bilinear  form on $A\oplus A^*$, which gives natural pairing, has degree zero (in non-shifted convention), the $A^*$ in the above definition is shifted by $1-d$ in order the corresponding cyclic Calabi-Yau structure on $A\oplus A^*$ is of degree $d$.

 The most simple example of pre-Calabi-Yau structure demonstrates that this structure does exist on any associative algebra. Namely, the structure of associative algebra on $A$ can be extended to the associative structure on $A\oplus A^*[1-d]$ in such a way, that the natural inner form is (graded)cyclic with respect to this multiplication. This amounts to the following fact: for any $A$-bimodule $M$ the associative multiplication on $A \oplus M$ is given  by
$ (a+f)(b+g)=ab+af+gb.$ In this simplest situation both structures on $A$ and on $A\oplus A^*$ are in fact associative algebras.
More examples one can find, for example, in \cite{I},
  \cite{DK}.

 Note that the notion of pre-Calabi-Yau algebra introduced in
\cite{KV}, \cite{Seidel}, \cite{TZ}, as an $A_{\infty}$-atructure on $A\oplus A^*$,  uses the fact that $A$ is finite dimensional, since there is no natural grading on the dual algebra $A^*={\rm Hom}(A,\K)$, induced form the grading on $A$
 in  infinite dimensional case.
One can reformulate it to give the general definition via the higher cyclic Hochschild complex (see \cite{KV}, \cite{Ktalk}), not requiring any finiteness conditions. 

This reformulation is based on the fact that due to the cyclic invariance of the natural (evaluation) pairing on $A\oplus A^*$, any tensor of the type $C_1\oo...\oo C_k \oo A$, where $C_i=A$ or $A^*$, can be considered as element of $\Hom_{\K} (C_1^*\oo...\oo \C_k^*, A)$. Generally, in linear map from one tensor
$C_1\oo ...\oo C_{l}$ to another $C_{l+1}\oo ...\oo C_{k+1}$ input from $A$ can be made into output from $A^*$ and vice versa.
The obtained in this way new definition will be given in the next section.
It is equivalent to the  definition above, where the ${\rm Hom}(A,\K)$  considered as graded Hom:
  $A^*=\oplus (A_n)^*=\underline{{\rm Hom}}(A,\K)$, in case the graded components of $A$ are finite dimensional.

We also should remark here that the theorem saying that pre-Calabi-Yau algebras give rise to TQFTs, analogous to the one proved in \cite{TZ}, for the definition~\ref{pcy1}, holds also for the definition via the higher cyclic Hochschild complex.

\section{Higher cyclic Hochschild complex}\label{hH}

We start with  the definition of the
{\it higher cyclic Hochschild cochains}  and {\it generalised necklace bracket}.

First let us consider higher Hochschild cochain complex $\hH$:
$$
C^{(N),n}(A):= \bigoplus\limits_{{r_1,...,r_N \geq 0}\atop{r_1+...+r_N=n}}
{\rm Hom}_{A^{\oo N} - mod - A^{\oo N}} (A^{\oo i_1}\oo...\oo A^{\oo i_N},  A^{\oo N}_{cycl} )
$$
The complex $C^{(N)}= \prod\limits_{n} C^{(N),n}$ is defined as $A^{\oo N}$-bimodule $\Hom$ from the $N$th power of the bar complex $\B(A)$:
$$C^{(N)}={\rm Hom}_{A^{\oo N} - mod - A^{\oo N}}(\cB^{\oo N}, A^{\oo N}_{cycl}) $$
to the  ${A}^{\oo N}$-bimodule  $A^{\oo N}_{cycl}$ with the following bimodule structure.
For any $x_1\oo...\oo x_N \in{A}^{\oo N }_{cycl} $ and elements $a_1\oo...\oo a_N, b_1\oo...\oo b_N \in{A}^{\oo N }$,
$$(a_1\oo...\oo a_N)\bullet (x_1\oo...\oo x_N)\bullet(b_1\oo...\oo b_N)= a_1x_1b_2 \oo ... \oo a_N x_N b_1.$$
Now we define the higher cyclic Hochshild complex $\chH$.

\begin{definition}\label{hh} For $N\geq 1$ the space of {\it N-higher cyclic Hochschild cochains} is defined as

$$
C^{(N)}_{cycl}(A):=
\prod_{r_1,...,r_N \geq 0}
{\rm Hom}_{A^{\oo N} - mod - A^{\oo N}}
(\mathop{\otimes}\limits_{i=1}^{N} A^{\oo r_i}, A^{\oo N}_{cycl})^{\Z_N},
$$
\end{definition}
The differential is coming from the bar  complex of $A^{\oo N}$-bimodules, after it is dualised  by \\
 ${\rm Hom}_{
A^{\oo N}-mod-A^{\oo N}}(-,A_{cycl}^{\oo N})$. It will be written precisely in section~\ref{sH}.

Elements of
${\rm Hom}(\mathop{\otimes}\limits_{i=1}^{N} A^{\oo r_i}, A^{\oo N}_{cycl})$
can be obviously   interpreted as
collections of $N$  operations with one output each.
This interpretation is obtained if we first pass to the isomorphic complex over $\K$. Then elements of ${\rm Hom}_{\k} \prod\limits_{r_1,...,r_N\geq 0}(A^{\oo r_i-2}, A_{cycl}^{\oo N})$ are interpreted as a collection of $N$  operations.

In this interpretation it is easy to see that there is a natural $\Z_N$ group action on ${\Hom}(\mathop{\otimes}\limits_{i=1}^{N} A^{\oo r_i}, A^{\oo N}_{cycl})$, which cyclically permutes operations and assigns a sign $(-1)^{(d-1)(N-1)} $, according to the Koszul rule (and taking into account degree of operation). 

The difference between $\hH$ and $\chH$ is that the latter consists only of elements invariant under this $\Z_N$ action, which we denote ${\rm Hom}(\mathop{\otimes}\limits_{i=1}^{N} A^{\oo r_i}, A^{\oo N}_{cycl})^{Z_N}.$
Let us introduce new notation:
 $C^{(N,d)}= C^{(N)}_{cycl}(A)\subset  C^{(N)}(A)$,  $\,\, C^{(N,d)}= C^{(N)}_{(d-1)(N-1)({\rm mod } 2)},$ where
$C^{(N)}_0$ are cochains symmetric under cyclic permutation of operations,
and $C^{(N)}_1$ are antisymmetric cochains.
By this we stress that $\Z_N$-invariant elements of $\hH$, constituting $\chH$, consist of either symmetric or antisymmetric cochains, depending on $N$ and $d$.
The fact that we should take only $\Z_N$-invariant elements comes from the condition of cyclic symmetry of the $A_{\infty}$-structure on $A\oplus A^*[1-d]$ w.r.t. the natural pairing.

We can point out at the difference between $\hH$ and $\chH$ also by saying
 that elements of ${\Hom}(\mathop{\otimes}\limits_{i=1}^{N} A^{\oo r_i}, A^{\oo N}_{cycl})$ have a fixed point, corresponding to the operation from which we start, however in   ${\Hom}(\mathop{\otimes}\limits_{i=1}^{N} A^{\oo r_i}, A^{\oo N}_{cycl})^{\Z_N}$, after it is symmetrized, that is invariants are taken, the fixed (starting) point does not exist any more.

Denote by $\chH=C^{(\bullet)}_{cycl}(A)= \mathop{\prod}\limits_{N \geq 1}  C^{(N)}_{cycl}(A)$ the space of all higher cyclic Hochschild cochains. Further throughout the paper we frequently omit the subscript $cycl$, if it does not produce any confusion.



 The space of all higher cyclic  Hochschild cochains is denoted by $C^{(\bullet)}_{A}(A)$  or $C^{(\bullet)}_{\k}(A)$, depending on whether we are dealing with  $A^{\oo N}$ - bimodule Homs, or consider corresponding $\k$-module Homs. Sometimes we omit the $\k$ when it is clear from the context.

Note, that  $C^{(1)}_{\k}(A)$
 is the space of usual
  Hochschild cochains.
 To set up notations remind that the usual Hochschild cochains are:

$$
C^{(1)}(A):= \prod_{r}
{\rm Hom}_{A - mod - A}(A^{\oo r}, A):
$$

$$
A \to \Hom_{A-mod-A} (A,A) \mathop{\to}\limits^D \Hom_{A-mod-A} (A^{\oo 2}, A) \to ...,
$$

where $A$ is sitting in degree $0$, $\Hom (A,A)$ in degree $1$, etc. and
$$(D_{\k} h)(v_1\oo ...\oo v_{n-1}) = v_1h(v_2\oo...\oo v_{n-1})-h(v_1v_2\oo v_2...\oo v_{n-1})+...$$

$$(-1)^{n-2} h(v_1\oo...\oo v_{n-2}v_{n-1})+(-1)^{n-1}  h(v_1\oo...\oo v_{n-2})v_{n-1}.$$

Now, when we start to define a Lie bracket on the higher cyclic Hochschild complex $C^{(\bullet)}_{A}(A)$,
 it becomes important  which shifts of the grading on $A$ we chose, so we consider the higher cyclic Hochschild complex with the following shifts:
$$
\chH [1]=C^{(N)}_{cycl}(A)[1]:= \prod_{r_1,...,r_N \geq 0}
{\Hom}_{A^{\oo N} - mod - A^{\oo N}}
(\mathop{\otimes}\limits_{i=1}^{N} A[1]^{\oo r_i}, A[1]^{\oo N}_{cycl})^{\Z_N}[1].
$$

 \begin{definition}\label{neckl} The {\it generalized necklace bracket} between two elements $f,g \in C_{cycl}^{(N)}(A)$ is given as
  $[f,g]_{g.n}= f\circ g - (-1)^{\sigma} g \circ f,$ where composition $f\circ g$ consists of inserting all outputs of $g$ to all inputs from $f$ with signs assigned according to the Koszul rule.
  \end{definition}

The composition for the generalised necklace bracket can be graphically depicted as follows:
\vskip1cm

\phantom0

\phantom0\hskip2cm\includegraphics[scale=0.8]{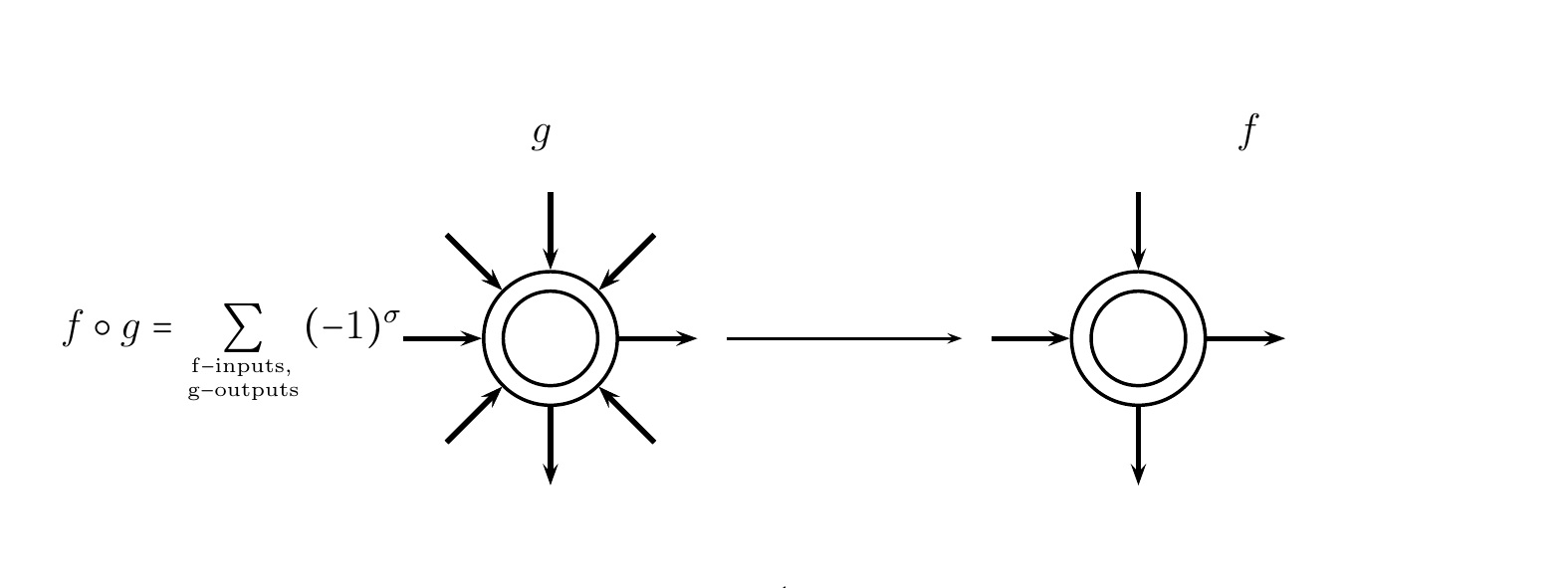}
\hfill

\phantom0



 Note that the (starting) point is not fixed in the elements of our complex (operations), thus
   generalized necklace bracket should produce also operations without a fixed point. Thus we need to clarify what means
   'insertion'  of one operation into another in the cyclic situation of definition above:
     we should think of insertion of operations with fixed point according to the above rule, and then symmetrizing the result, by taking
   each resulting operation with all possible fixed points to
   the output.

Since the defined above composition $f\circ g$ makes $\chH[1]=C_{cycl}^{(\bullet)}[1]$ into a graded pre-Lie algebra, the generalized necklace bracket obtained from it as a graded commutator, makes $C_{cycl}^{(\bullet)}[1]$ into a graded Lie algebra.
We denote it by ${\bf g}=(C_{cycl}^{(\bullet)}(A)[1], [,]_{g.n}).$

\begin{definition}\label{pcy2}
The pre-Calabi-Yau structure on $A$ is an element $m=\mathop{\sum}\limits_{N\geq 0} m^{(N)}$, $m^{(N)} \in C^{(N,d)}(A)[1]$
from the space of the $\Z_N$-invariant higher cyclic Hochschild cochains
$$\chH[1]=C^{(N,d)}[1] =\prod\limits_{r_1,...,r_N \geq 0}
 {\Hom}_{A^{\oo N} - mod - A^{\oo N}}
(\mathop{\otimes}\limits_{i=1}^{N} A[1]^{\oo r_i}, A[1]^{\oo N}_{cycl})^{\Z_N}[1],$$
   which is  a solution to the Maurer-Cartan equation $[m,m]_{g.n}=0$ with respect to generalised necklace bracket.
\end{definition}

Since in shifted by 1 situation operations have degree one, those elements of the $\chH [1]$ which correspond to $d$-pre-Calabi-Yau structure should have degree $(d-1)(N-1)$.
While  reformulating the definition  we take  into account that in the  definition~\ref{pcy1}  $A^*$ is shifted by $1-d$.

For the sake of clarity,  simplicity, and  since the main formality result holds only in this situation, we mainly consider here the grading, where $A$ is sitting in degree zero: $A_0=A$. This prompts us to deal with 2-pre-Calabi-Yau structures.



In the case $n=2, N=2$, which will be interesting for us in the next section, the definition says that we should take antisymmetric elements of $\hH$ into $\chH$ complex.

\section{Double Poisson bracket and the Maurer-Cartan equation}\label{dP}

In this section we discuss a correspondence between particular part of pre-Calabi-Yau structure and the structure of double Poisson bracket invented by Van den Bergh \cite{VdB} as a structure which produces the Poisson bracket on representation spaces.

There were many efforts to construct a reasonable notion of noncommutative Poisson bracket, which would according to the ideology of \cite{Kf, KR} induce a kind of  Poisson bracket on representation spaces or their moduli. First, the noncommutative Poisson bracket was defined in an obvious way: the same way as it is done in the commutative case, as a bracket on $A$, $\{-,-\}: A \to A$ which satisfy the Leibnitz rule: $\{a,bc\}=b\{a,c\}+\{a,b\}c$. But this notion turns out to be too restrictive, in \cite{FL} it was shown that defined this way bracket on noncommutative prime rings can be only commutator bracket $[a,b]=ab-ba$. There was a notion of 'noncommutative Poisson structure' introduced in \cite{X, BG}, but it is only known that this bracket on the center of $A$ produces the usual commutative Poisson structure, it is unclear what it gives on moduli of representations. Then the attempts to introduce the notion of noncommutative Poisson bracket lead to a good definition in \cite{VdB, CB}. In \cite{VdB} the notion of the double Poisson bracket was defined as a map $\dd \cdot, \cdot \c: A\oo A \to A\oo A$, satisfying the axioms which are certain generalization (thickening) of the usual Poisson axioms of anti-symmetry, Leibnitz and Jacobi identities.
 In \cite{CB} the noncommutative Poisson structure, called $H_0$-Poisson structure was defined as a Lie bracket on zero Hochschild homology of $A$: $H_0=A/[A,A]$, such
that the map $\{ \bar a,- \}: A/[A,A] \to  A/[A,A] $ is induced by a derivation $d_a: A \to A$. These explained many effects, for example, clarified the study of quasi-poisson structures \cite{AYM}.
There were further developments like \cite{R}, but we are trying here to continue the line of initial ideas from \cite{Kf, KR}.

We will put the earlier approaches in a more general framework, which  explains a pattern of this generalisation (thickening) process from the perspective of the whole pre-Calabi-Yau structure. For example, the twisted structure of the diagonal bimodule $A^{\oo N}$ will show how multiple derivations associated to a noncommutative polyvector fields interfere.
We will  see also in this section how the 'double' definition comes as a particular part of  general situation, coming from pre-Calabi-Yau structure. Namely, we will demonstrate that double bracket is defined by $\xi\delta$-words with two $\delta$th and no $\xi$th.


Remind, that double Poisson  bracket is defined as a map $\dd \cdot, \cdot \c: A\oo  A \to A\oo A$ satisfying the following axioms:

Anti-symmetry:

\begin{equation}\label{sym}
\dd a,b \c = - \dd b,a \c^{op}
\end{equation}
Here $\dd b,a \c^{op}$ means the twist in the tensor product, i.e. if $\dd b,a \c =
\sum\limits_i b_i \oo c_i$, then $\dd b,a \c^{op}=\sum\limits_i c_i \oo b_i.$

Double Leibniz:

\begin{equation}\label{Leib}
\dd a,bc \c = b \dd a,c \c + \dd a,b \c c
\end{equation}

(here we use an outer bimodule structure on $A \oo A: \quad a(b\oo c)=ab \oo c, (a\oo b)c=a\oo bc),$
and double Jacobi identity:

\begin{equation}\label{Jac}
\dd a, \dd b,c \c\c_L + \tau _{(123)} \dd b, \dd c,a \c \c_L + \tau _{(132)} \dd c, \dd a,b \c \c_L
\end{equation}
Here for $a \in A\oo A \oo A, \,\,$ and $ \sigma \in S_3\,\,$
$$\tau _{\sigma}(a)=
a_{\sigma^{-1} (1)} \oo a_{\sigma^{-1} (2)} \oo a_{\sigma^{-1} (3)}.$$
The $\dd\,\c_L$ defined as
$$
 \dd b, a_1\oo...\oo a_n \c_L = \dd b, a_1 \c \oo a_1\oo ... \oo a_n
$$

The connection between the two structures is described by the following theorem.

\begin{theorem}\label{main0}
Let we have $A_{\infty}$-structure on $(A\oplus A^*,  m=\sum\limits_{i=2, i\neq 4}^{\infty}m_i^{(1)})$.
Define the bracket by the formula
$$
(*)\quad \quad \langle g\otimes f,\dd b,a\c\rangle:=\langle m_3(a,f,b), g \rangle,
$$
where $a,b\in A$, $f,g\in A^*$ and $m_3(a,f,b)=c\in A$ corresponds to the component of  solution to  the Maurer-Cartan $m_3$: $A\times A^*\times A\to A$ corresponding to the cyclic tensor $A\otimes A^*\otimes A\otimes A^*$. Then this bracket does satisfy all axioms of the double Poisson algebra.

Moreover, pre-Calabi-Yau structures
corresponding to the cyclic tensor
$A\otimes A^*\otimes A\otimes A^*$
  with $m_i=0, i\geq 4$ are in the bijective correspondence defined by $(*)$ with the double Poisson brackets
for an  arbitrary associative algebra $A$.
\end{theorem}

The detailed proof of this theorem, taking into account signs and other details, was given in \cite{IK} (see also \cite{IKihesP}) in terms of
definition~\ref{pcy1} of pre-Calabi-Yau structure.

We present here main idea of this proof, using  definition~\ref{pcy2} via higher cyclic Hochschild complex. It looks very transparent this way, which emphasises another advantage of  this definition.

In terms of  definition~\ref{pcy2} the Maurer-Cartan equation on 'invariant' with respect to the action of cyclic group elements from
the higher cyclic Hochschild complex of particular kind, described in the theorem looks like:

\vskip1cm

\phantom0\hskip2cm
\includegraphics[scale=0.7]{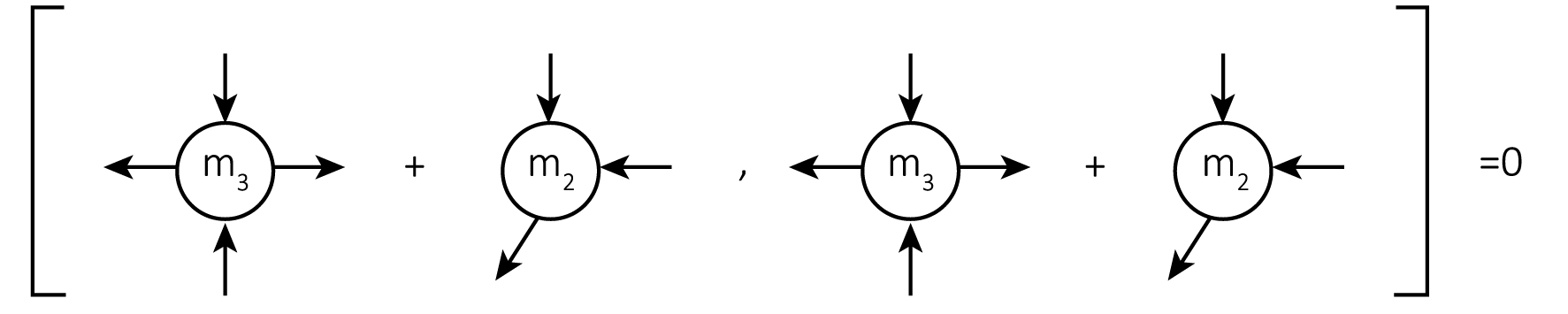}
\hfill

\phantom0

\vskip1cm

Hence  the  Maurer-Cartan (in appropriate arity)  is equivalent to the following equations:

\vskip1cm

\phantom0\hskip4cm
\includegraphics[scale=0.7]{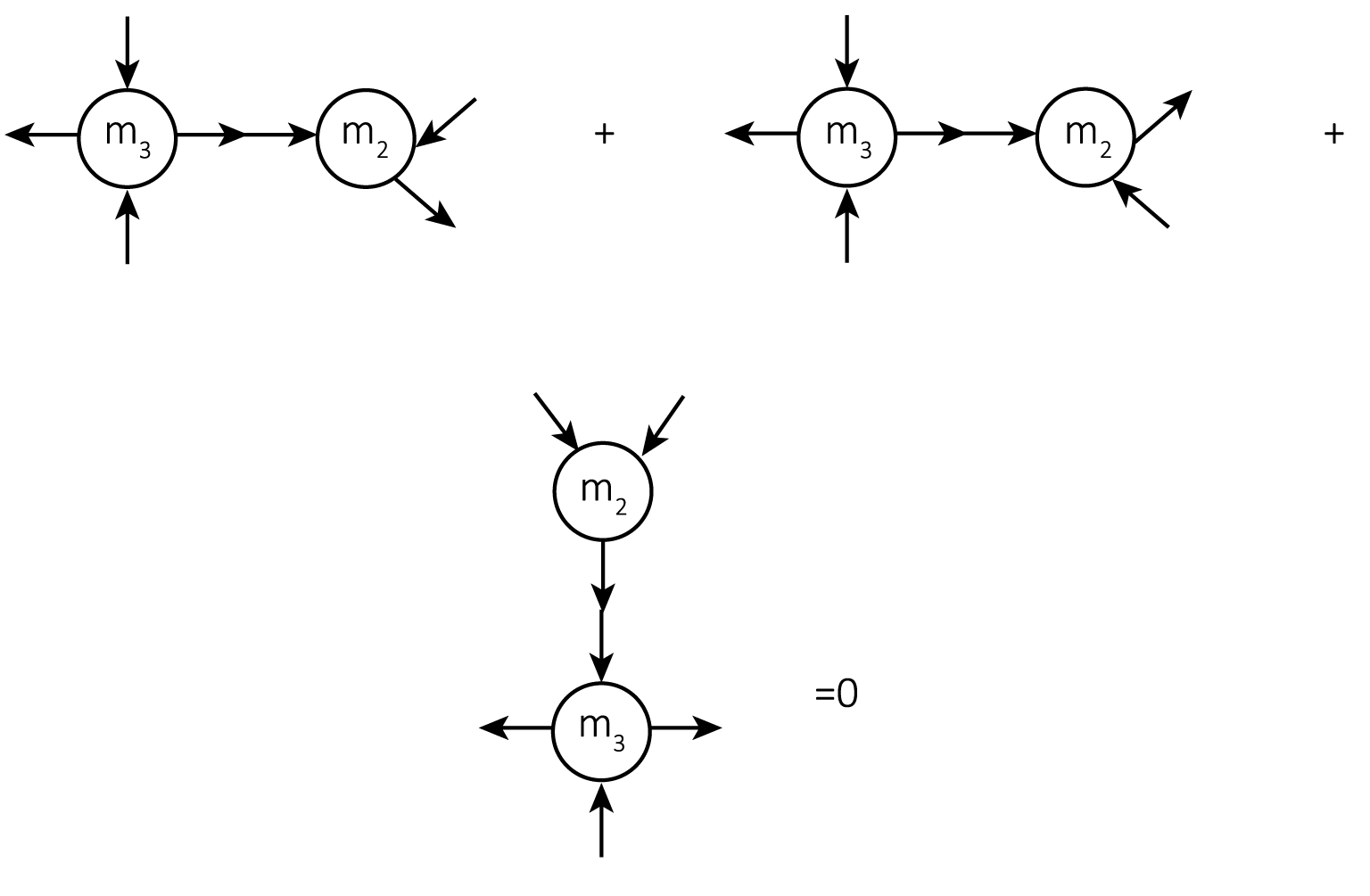}
\hfill

\phantom0

\vskip1cm

\phantom0\hskip4cm
\includegraphics[scale=0.2]{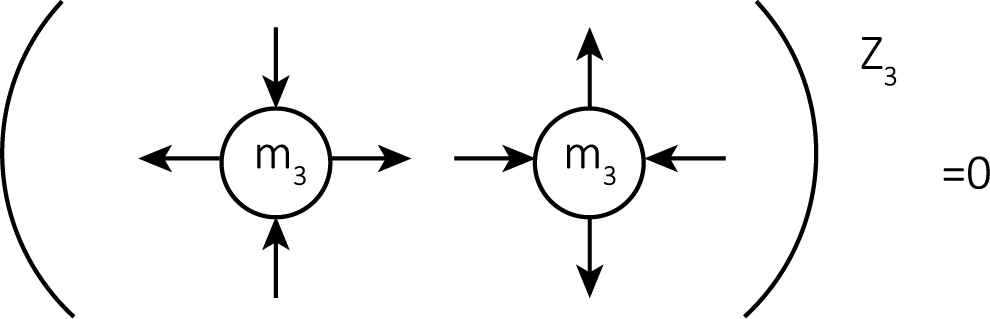}
\hfill

\phantom0

These two are clearly exactly the double  Leibnitz and double Jacobi identities respectively. The thing to be checked now is related to the following fact. The element of the higher cyclic Hochschild complex with two inputs and two outputs canonically corresponds (via the pairing on $A\oplus A^*$) to the maps $D: A\times A \to A\times A $ and $M: A\times A^*\times A \to A$. This correspondence is defined however only  up to an arbitrary permutation of terms $A$. To establish an isomorphism between the two structures, we then need to choose appropriately this correspondence, which is done by formula (*) in \cite{IK}.
After that axioms of double Poisson bracket can be checked, taking into account signs. Moreover we  ensure that no other axioms appear from the Maurer-Cartan equation in case of the structure $(A\oplus A^*, m_2+m_3)$, hence all double Poisson brackets can be obtained from these structures, that is the map defined by (*) is a surjection. This means that structures of mentioned type are indeed in a bijective correspondence with the double Poisson brackets.

The question on better understanding of the structure
needed to produce Poisson bracket on representation spaces recently received  much attention. It is discussed, apart from mentioned paper \cite{IKihesP}, for example, in \cite{Art, VayVay}.
Our work \cite{IK}, \cite{IKihesP} was also extended to the case when $A\neq A_0$, by means of the same correspondence given by formula (*), with signs added in \cite{H}.



\section{Small subcomplex of the higher cyclic Hochshild complex}\label{sH}

Here we work with the  smooth algebras \cite{KR}, which are finitely generated 2-formally smooth algebras.  By {\it 2-formally smooth} algebra we mean a formally smooth algebra in the sense of
 J. Cuntz and D. Quillen \cite{CQ}:

 \begin{definition}
An algebra A is 2-formally smooth (=quasi-free)
if and only if it satisfies one of the following equivalent properties:

(1) (Lifting property for nilpotent extensions) for any algebra $B$, a two-sided nilpotent
ideal $I \in B (I = BIB, I^n = 0$ for $n \gg 0)$, and for any algebra homomorphism
$ f: A \to B/I$, there exists an algebra homomorphism $\tilde f: A\to B$ such that $f = pr_{B \to B/I} \circ \tilde f$ is a natural projection.

(2) $Ext^2_{A-mod-A}(A,M)=0$
 for any bimodule $M \in A-mod-A$.

(3) The A-bimodule
$\Omega_A^1 = {\rm Ker} (m_A:  A \oo A \to A)$ is projective.

\end{definition}


First, consider a subcomplex $\zeta$ of the higher cyclic  Hochschild complex, which we define as follows. Take a quotient complex
$$
R_{\rm min}=[0\to \Omega\to A\otimes A]\twoheadrightarrow A
$$
of the bar complex (considered as a complex of $A$-bimodules). Namely, $R_{\rm min}={\cal B}/{\cal F}$, where ${\cal B}$ is
the bar complex
$$
{\cal B}=[\dots A\otimes A\otimes A\mathop{\to}\limits^{D_3} A\otimes A]\twoheadrightarrow A.
$$
Denote its usual differential by $D_A$ or just $D$, when it is clear that we are talking about complex of $A$-bimodules.
 Obviously, $\cB$ as well as $\cR_{min}$ provide a free $A$-bimodule resolution for the diagonal bimodule $A$.

  Let $\cal F$ be the subcomplex generated by $A^{\otimes k}$ with $k\geq 4$ and $\ker D_3$, i.e. ${\cal F }=\oplus A^{\otimes \geq 4}\oplus {\rm Ker} D_3$. Note that $\Omega=A^{\otimes 3}/\ker D_3$
is isomorphic to the kernel of the multiplication map $\mu:A\otimes A\to A$. We
equip $R_{\rm min}$ with the grading for which $R_{-1}=\Omega$ and $R_0=A\otimes A$. Thus  we have a resolution
 $R_{\rm min}\in{\rm Compl}(A^{\rm e}-mod)$ of a diagonal bimodule $A$.

Then  consider $N$th tensor power of $R_{\rm min}$:
$$
R_{\rm min}^{\otimes N}\in{\rm Compl}((A^{\rm e})^{\otimes N}-{\rm mod}),
$$
and dualise it by taking  Hom to an $A^{\oo N}$-bimodule $A^{\oo N}_{cycl}$ with the defined above structure:
$$
{\rm Hom}_{(A^{\rm e})^{\otimes N}}(R_{\rm min}^{\otimes N},A^{\oo N}_{{\rm cycl}})=:\zeta^{(N)}.
$$

For $N=1$ applying the functor ${\rm Hom}_{A^{\rm e}}(-,A)$ to $R_{\rm min} \in {\rm Compl}(A^e-{\rm mod})$ we get a subcomplex $\zeta = {\rm Ann} ({\cal F})={\rm Hom}_{A^{\rm e}}(R_{min},A)$ of the usual Hochschild complex
$C^{\bullet}(A,A)=C^{(1)}(A)=
{\rm Hom}_{A^{\rm e}}({\cal B},A)$:
$$
C^{\bullet}(A,A) \supset  {\rm Hom}_{A^{\rm e}}(R_{\rm min},A),
$$
where
$$
{\zeta}={\rm Ann}({\cal F})=\{\Phi\in C^\bullet(A,A)[1]:\Phi(h)=0\ \text{for}\ h\in{\cal F}\}.
$$
Thus

$\zeta={\rm Ann}({\cal F})=\{\text{chains in}\ {\rm Hom}_{A^e}({\cal B},A),\ \text{turning}\ {\cal F}\ \text{into $0$}\}$

$=\{\Phi(a_1\otimes\dots\otimes a_n)$, s.t.
$\Phi(a_1\otimes\dots\otimes a_n)=0$, $n>3$ and

$\Phi(a_1\otimes a_2\otimes a_3)=0$ iff $a_1\otimes a_2\otimes a_3\in \ker D_3={\rm Im}D_4$\}.

Thus, $\Phi\in \Hom_{A^{e}}(A^{\otimes 3},A)$ is in $\zeta$ if and only if it is an $A$-bimodule derivation, that is satisfies the Leibnitz rule:
$$
\Phi(a_1\otimes a_2a_3\otimes a_4)=\Phi(a_1\otimes a_2\otimes a_3a_4)-\Phi(a_1a_2\otimes a_3\otimes a_4).
$$

Note that ${ \Hom}_{A^{\rm e}}(A\otimes A,A)$ is naturally isomorphic to $A$, while ${\Hom}_{A^{\rm e}}(\Omega,A)$ is
naturally identified with ${\rm Der}_{A^e}(A^{\otimes 3},A)$, which interprets the complex $\zeta_A=\Hom_{A^e} (R_{min},A)$

$$0\gets \Hom_{A^e}(\Omega,A) \mathop{\gets}^{D_3^*} \Hom_{A^e}(A\otimes A, A) \mathop{\gets}^{D_2^*}\Hom_{A^e}(A,A )$$

 as
$$
0  \gets {\rm Der}_{A^e}(A^{\otimes 3},A) \mathop{\gets}^{D_3^*} A\mathop{\gets}^{D_2^*} \K.
$$
We can pass from ${\rm \Hom}_{A^{\rm e}}$ to $\Hom_{\K}$, and since
$$\Hom_{A^{\rm e}}(A^{\otimes n+2})\simeq \Hom_\K(A^{\otimes n},A), \quad
\Hom_{A^{\rm e}}(A\otimes A,A)\simeq A,\quad
\Hom_{A^{\rm e}}(\Omega, A) \simeq {\rm Der}_{\k}(A)\, \subset \Hom_{\K}(A,A)
$$

we have an isomorphic complex over $\k$:
$$
\zeta_{\k}: \quad 0\gets {\rm Der}_{\k}\,A\mathop{\gets}^{d_3^*} A
$$
where ${\rm Der}_{\k}\,A$ is the space of derivations of $A$ from $\Hom_\K(A,A)$.

Cohomologies of the usual Hochschild cochain complex:
$$
H^{\bullet}(C^{(1)})(A)=Ext_{A-mod-A}^{\bullet}(A,A)
$$
are the Hochschild cohomologies.

 Analogously,  denote cohomologies of $N$th slice the $\chH$ complex $C^{(N)}$ by
$$
H^{\bullet}(C^{(N)}(A))=Ext_{A^{\oo N}-mod-A^{\oo N}}^{\bullet}(A^{\oo N}, A^{\oo N}_{cycl}),
$$
and cohomologies of the whole $\chH$ complex $C^{(\bullet)}$ by $H^{\bullet}C^{(\bullet)}.$

We explain here in more details what we do in the case of arbitrary $N$.
The $\chH$ complex (before taking invariants) is the $N$th tenzor power of the bar complex $\cB^{\oo N}$ dualised by $\Hom_{A^{\oo N}-mod-A^{\oo N}}(-, A^{\oo N}_{cycl}):$
$$ C^{(N)}(A)=\Hom_{A^{\oo N}-mod-A^{\oo N}}(\cB^{\oo N}, A^{\oo N}_{cycl}).$$

The complex $\cR_{min}$ is obviously a quotient complex of $\cB^{\oo N}$. Indeed,  $R_{min}^{\oo N}=({\cal B}/{\cal F})^{\oo N}={\cal B}^{\oo N}/{\cal J}$
where ${\cal J}=\cF \oo \cB^{\oo (N-1)} + \cB \oo \cF \oo \cB^{\oo (N-2)}+...+  \cB^{\oo (N-1)} \oo \cF$.
Here we need to check of course that $\cJ$ is a submodule in $A^{\oo N}$-bimodule $\cB ^{\oo N}.$

To obtain the complex $\zeta^{(N)}$ we take $N$th tensor power of small complex $R_{min}$ and dualize it by ${\rm Hom}_{A^{\oo N}-mod-A^{\oo N}}(-, A^{\oo N}_{cycl})$. Then we take invariants under $\Z_N$ action.

The structure of $A^{\oo N}$  bimodule on ${\cal B}^{\otimes N}$ is natural, on $A^{\oo N}_{cycl}$ as defined above.

Thus $\zeta^{(N)}_A=\Ann \cJ=
\Hom_{A^{\oo N}-mod-A^{\oo N}} (R_{min}^{\oo N}, A^{\oo N}_{cycl})^{\Z_N} \subset
\Hom_{A^{\oo N}-mod-A^{\oo N}} (\cB^{\oo N}, A^{\oo N}_{cycl})^{\Z_N}.$

We can describe this annihilator as
$$\zeta^{(N)}=\Ann \cJ=
\{ \Phi \in \Hom_{A^{\oo N}-mod-A^{\oo N}} (\cB^{\oo N}, A^{\oo N}_{cycl}))^{\Z_N} \,|\, $$

$$\Phi({\cB}^{\oo r} \oo  {\cF} \oo {\cB}^{\oo s})=0, \quad \forall r+s=N-1\}.$$

This means
$\zeta^{(N)}=\Ann \cJ$ formed by those element of vector space
$$
E^{\oo N}=\bigcap\limits_{s+r=N-1} \{ {\cB}^{\oo r} \oo \Ann {\cF} \oo {\cB}^{\oo s}  \},
$$
which are $A^{\oo N}$-bimodule morphisms.

This  leads us to the description of the small subcomplex $\zeta^{(N)}$ of the $\chH$ complex $C^{ (N)}(A)$ in terms of the appropriately chosen basis.

Starting from this place, when we choose a basis, we will deal with a free algebra $A$, in stead of general  smooth algebra.
The obvious free basis in $\Omega_{A-mod-A}$ consists of $dx_i=1\oo x_i- x_i\oo 1$. Denote the basis of free $A$-bimodule $A\oo A$ by $\xi$.
Elements of dual bases in $\Hom_{A-mod-A}(\Omega, A)$ and $\Hom_{A-mod-A}(A\oo A, A)$ denote  by
$\p_i$ and $\xi^*$: $\p_i(dx_j)=\delta_{ij} 1, \, \xi^*(\xi)=1$ respectively.
Corresponding bases of  $\Hom_{\k}(\Omega, A)$ and $\Hom_{\k}(A\oo A, A)$ are  $\{\p_i u, u\in \l X \r\}$ and $\{ \xi u, u \in \l X \r\}$ respectively.


The basis of the complex  $\zeta^{( N)}_{\k}=\Hom_{\k}(R_{\min}^{\oo N}, A^{\oo N}_{cycl})^{\Z_N}$
will consist of cyclic monomials on $\xi^*$, $\p_i$ and $x_i$ which we depict as follows (we will further write just $\xi$ in stead of $\xi^*$).

\phantom0

\vskip1cm

\phantom0\hskip3cm
\includegraphics[scale=0.2]{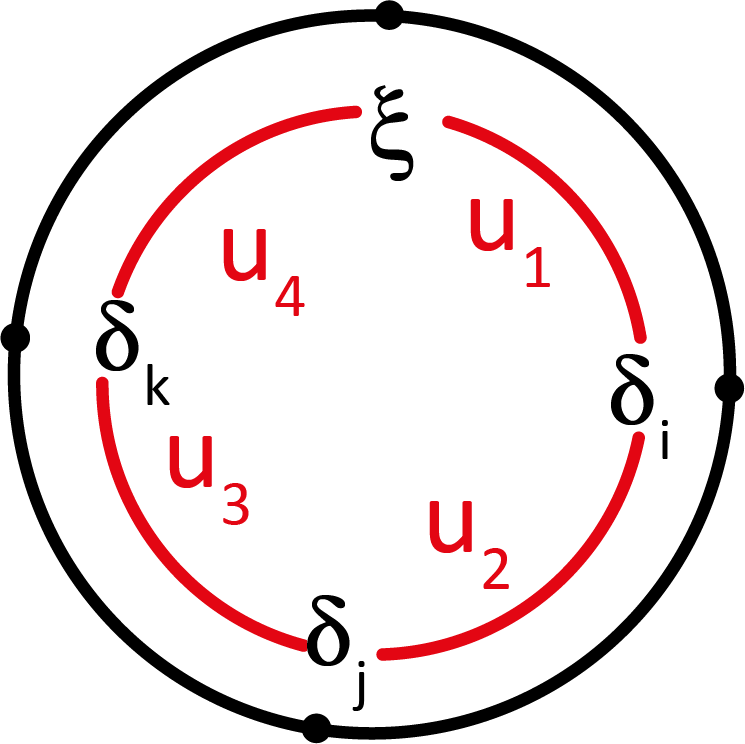}
\hfill

\phantom0

pic.1


Now the important for further considerations point is to embed
the subcomplex $ \zeta^{(N)}$ into the higher cyclic Hochschild complex $C^{(N)}$. We define this embedding  by specifying the
operation, i.e. an element of the higher cyclic Hochschild complex $\Phi_{\w} \in \mathop{\Hom}_\K(\mathop{\oo}\limits_{i=1}^N A^{\otimes n_i},A^{\otimes N})^{\Z_N}$, which corresponds to a given  $\xi\d$-{\it monomial} $\w$.
Here $n=\sum n_i$ is the $\delta$-degree of $\xi\d$-{\it monomial} $w$ and $N$ is its  $\delta,\xi$-degree.

Let $X$-monomials $u_1,\dots,u_n$ be an input  of $\Phi_\w$. The output will be a linear combination of tuples of monomials from
$A$ colored green in the following picture.
All circles and arcs in the picture are oriented clockwise, so  one can read outputs
following the orientation of the circles.
The sum in the  linear combination is over all 'gluings' of variables $x_i$ from the input monomials (black) with
$\delta_i$ (with the same index $i$) in the $\xi\delta$-monomial $w$.


The $\xi\delta$-monomial depicted below represent an operation $\Phi: A^{\oo 3} \to A^{\oo 5},$ (or more precisely:
$A\oo A^0 \oo A \oo A \oo A^0 \to A^{\oo 5}$) from the $\chH$ complex $C^{(N)}$ for $N=5$.

\vskip1cm

\phantom0\hskip4cm
\includegraphics[scale=0.06]{Modpicop.png}
\hfill

\phantom0

pic.2


In terms of the above $\xi\p$-basis we now describe differentials $D^*_A=D^*$  and $D^*_{\k}=d^*$ in dualized complexes.

Let us spell out first the usual differential $D_A$ on one copy of the bar complex $\cB$:
$$D_A(u_1\oo ...\oo u_n)=u_1u_2\oo...\oo u_n-u_1\oo u_2u_3\oo...\oo u_n+...+(-1)^{n} u_1\oo...\oo  u_{n-1}u_n.$$

After we dualise this complex by $\Hom_{A-mod-A}(-,A)$,
we get a usual dual differential $D^*_A$, $D^*f(u)=f(Du)$ on the Hochschild complex :
$$(D^*_A f)(u_1\oo ...\oo u_{n+1})=f(u_1u_2\oo...\oo u_{n+1})-f(u_1\oo u_2u_3\oo...\oo u_{n+1})+...+(-1)^{n+1} f(u_1\oo...\oo u_n u_{n+1}),$$
where $f\in \Hom_{A-mod-A}(\cB,A).$

When we pass to $\Hom_{\k}(\cB, A)$, since $f$ is an $A$-bimodule morphism, an element $h \in \Hom_{\k}(\cB, A)$ is defined by
$h(v_1\oo ...\oo v_n)=f(1\oo v_1 \oo...\oo v_n\oo 1),$
and
$$(D^*_{\k} h)(v_1\oo ...\oo v_{n-1}) = v_1h(v_2\oo...\oo v_{n-1})-h(v_1v_2\oo v_2...\oo v_{n-1})+...$$

$$(-1)^{n-2} h(v_1\oo...\oo v_{n-2}v_{n-1})+(-1)^{n-1}  h(v_1\oo...\oo v_{n-2})v_{n-1},$$
where $h \in \Hom_{\k}(\cB, A).$

Doing the same for the
 tensor product of $N$ copies of the bar complex $\cB^{\oo N}$ and dualising it by
 $\Hom_{A^{\oo N}-mod-A^{\oo N}}(-,A^{\oo N}_{cycl})$, we obtain the expression for the differential in the higher cyclic Hochschild complex:
$$D^*h(v_1,...,v_N)= \sum\limits_{\alpha=1}^n (-1)^{s_1+...+s_{\alpha-1}} D^{*\alpha} h(v_1,...,v_N),$$
where $v_{\alpha}=x_1^{\a}...x_{s_{\a}}^{\a}\in A^{\oo s_{\alpha}}\subset \B$, and

$$D^{*\alpha}_{\K} h(v_1,...,v_N)\sum\limits_{j=1}^{s_{\alpha}-1} (-1)^j
h(v_1\oo ...\oo v_{\a-1} \oo x_1^{\a}\oo ... \oo x_j^{\a}x_{j+1}^{\a} \oo ...\oo   x_{s_{\a}}^{\a} \oo v_{\a+1}\oo ...\oo v_N) + $$

$$(1\oo...\oo 1\oo x_1^{\a}\oo...\oo 1)\bullet h(v_1\oo ...\oo v_{\a-1} \oo x_2^{\a}\oo ...\oo   x_{s_{\a}}^{\a} \oo v_{\a+1}\oo ...\oo v_N) + $$

$$(-1)^{s_{\a}} h(v_1\oo ...\oo v_{\a-1} \oo x_1^{\a}\oo ...\oo   x_{s_{\a-1}}^{\a} \oo v_{\a+1}\oo ...\oo v_N) \bullet
(1\oo...\oo 1\oo x_{s_{\a}}^{\a}\oo...\oo 1). $$

Here the  element $(1\oo...\oo 1\oo x_1^{\a}\oo...\oo 1) $ has $x_1^{\a}$ in the place $\a$, the element ($1\oo...\oo 1\oo x_{s_{\a}}^{\a} \oo...\oo 1)$ has $x^{\a}_{s_{\a}}$ in place $\a+1 ({\rm mod} N)$, and $\bullet$ stays for the multiplication in $A^{\oo N}$-bimodule  $A^{\oo N}_{cycl}$, as described in section \ref{hH}.



We seen tat $R_{min}^{\oo N}=\cB^{\oo N}/{\cal J}$, thus $\zeta=\Hom(R_{min}^{\oo N}, A^{\oo N}_{cycl})$   is a subcomplex of the $\chH$ complex
$\Hom(\cB^{\oo N}, A^{\oo N}_{cycl}).$  More precisely, in both complexex we should take $\Z_N$-invariant elements.


The subcomplex $\zeta $ is quasi-isomorphic to the $\chH$ complex since $A$ is smooth.

Remind that we denoted
$H^{\bullet}C^{(N)}(A)=Ext_{A^{\oo N}-mod-A^{\oo N}} (A^{\oo N}, A^{\oo N}_{cycl})$
homologies of $N$th slice of $\chH$ complex $C^{(N)}(A)$.

\begin{proposition} Let $A$ be a
smooth algebra (in particular, finitely generated free associative algebra). Then
$H^{\bullet}C^{(N)}(A,A)=H^{\bullet}\zeta^{(N)}$.
\end{proposition}

\begin{proof}
The statement follows from the fact that our algebra is smooth and both complexes comes from projective resolutions of a diagonal bimodule $A^{\oo N}$.

\end{proof}

One of our main goals will be to prove that these homologies are pure, namely concentrated in the diagonal of the defined above bigrading on the $\chH=C^{(N),n}$, and the complex is
$L_{\infty}$-formal. We believe that concrete calculations in \cite{OS} (check of the Jacobi identity, etc.) for noncommutative bivector fields (in sense of \cite{CB}) are easily explained, and perhaps inspired by our results \cite{IKihes} and are in fact performed in our basis of the small complex $\zeta$. The case of noncommutative projective space corresponds to the free path algebra of the Kronecker quiver $\Kc_n$, for which we prove the formality in section \ref{q}.


In the next section we consider the Lie structure on  $C^{(\bullet)}[1]$, described in section \ref{hH}, and prove that the subcomplex $\zeta^{(\bullet)}$ of the complex ${\bf g}=C^{(\bullet)}[1]$ is closed under the Lie bracket in ${\bf g}$. Thus it forms a Lie subalgebra ${\bf g_0} \subset {\bf g}$. Moreover, we show how this bracket is combinatorially described in terms of the basis in $\bf g_0=\zeta^{(\bullet)}$, consisting of $\xi\delta$-monomials.


\section{Lie bracket on $\zeta^{(\bullet)}$}\label{Lie}

We give here a constructive description of the bracket in the small subcomplex $\zeta_{\K}^{(N)}$ of the Hochschield complex in terms of $\xi\delta$ calculus which makes it into a Lie subalgebra of $C_{\K}^{(N)}$.

\begin{theorem}\label{THM1}
I. The above described embedding $\zeta_{\K}\to C^{(\bullet)}(A,A)[1]$ is an embedding of complexes, whose image is
a Lie subalgebra of ${\bf g}=C^{(\bullet)}(A,A)[1]$ equipped with the generalised necklace bracket.

II. Precise combinatorial description of this bracket is given by (*) (pic. 3).
\end{theorem}

\begin{proof}
To prove this we need to show that the bracket of the Lie algebra
${\bf g}=C^\bullet(A,A)[1]$ applied to
$\xi\delta$-monomials yields a member of $\zeta$, that is, a linear combination of
$\xi\delta$-monomials again. Here we demonstrate how it works.

Let $A$ and $B$ be two $\xi\delta$-monomials. We perform composition of corresponding operations $U\circ W$ from the Hochschild complex according to the necklace bracket rule. We will see that  we can not express the resulting operation $U\circ W$ via $\xi\delta$ monomials, but we can do it for the operation $[U,V]=U\circ W-W\circ U$. Perform first $U\circ W$. This composition of operations from $C^{(\bullet)}(A)$ is realised as application of $\xi\delta$-monomial $A$ to the input (according to the procedure described by pic.2), and then application of $\xi\delta$-monomial $B$ to the output of the first operation. As an output of this composition we will get linear combination of monomials  from $A^{\oo i}.$
We call such a monomial {\it non-essential} if it is obtained as a result of gluing some letter $x_i$ from the input of operation $A$ to some $\d_i.$ Gluing  letters $x_j$ from the  operation A itself (red arcs in pic.2) to the $\d_j $ will result in obtaining   {\it essential} monomials in the output of composition $U\circ W$.
Note that the copy of the same monomial present in composition $U\circ W$ can be essential or not, depending on how it is obtained. Thus to be essential is not a property of the monomial, but it just characterises the way it got into the output of the composition of these operations.

We claim that non-essential output monomials for the operations $U\circ W$ and $W\circ U$ will be the same (with the same coefficients) and therefore they will cancel out in the bracket $[U,W].$ This bracket is
formed only by  essential outputs meaning exactly that it is described by the  operation on $\xi\delta$-monomials, described in picture 3. Namely, letters $x_i$ from the monomial $A$ (red in pic. 3) are getting inserted in $\delta_i$ of the second $\xi\delta$ word $B$ (and the other way around for $W\circ U$). Insertions of this kind alone define $[A,B]$.

\vskip1cm

\phantom0\hskip4cm
\includegraphics[scale=0.2]{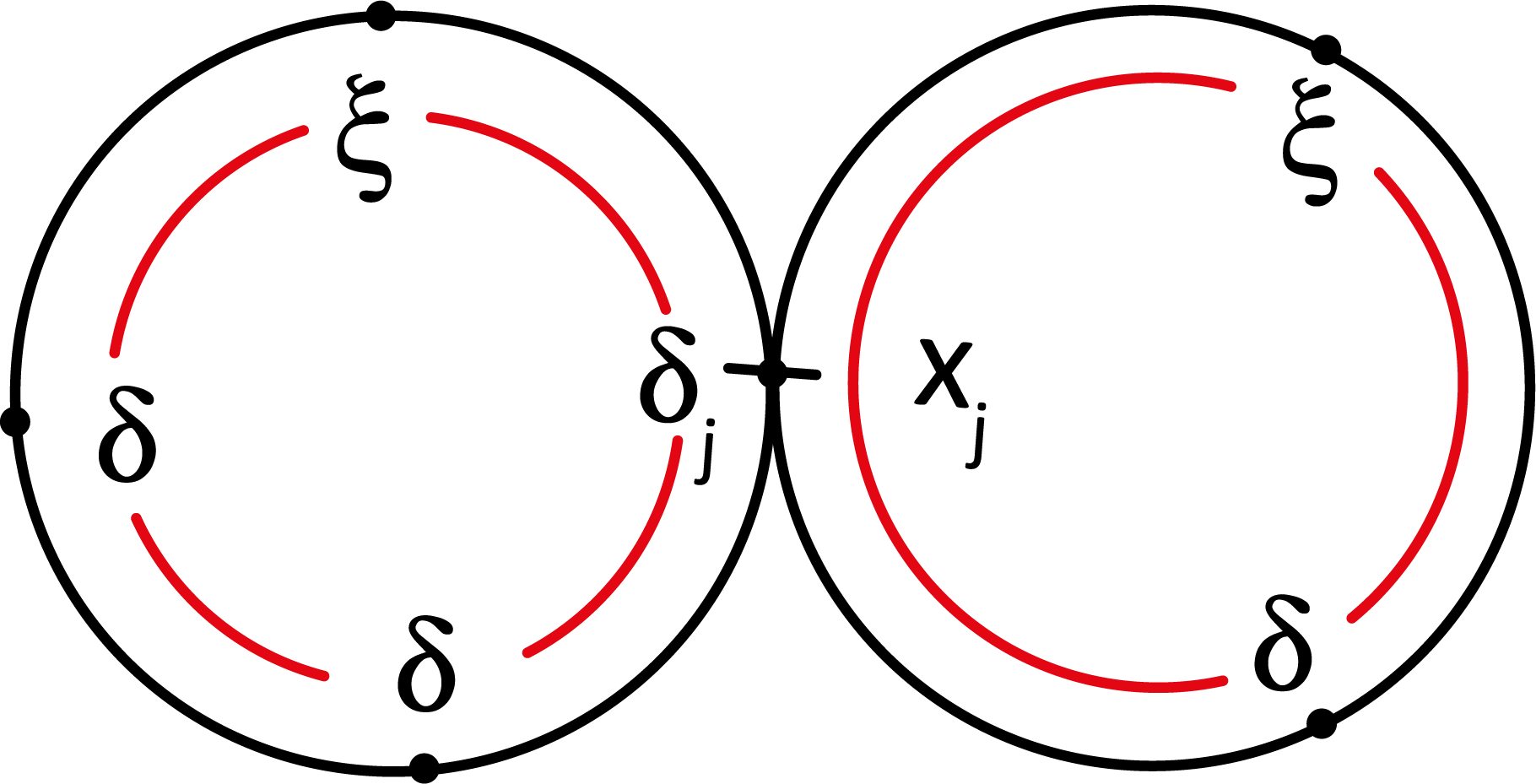}
\hfill

\phantom0

\vskip1cm

\phantom0\hskip4cm
\includegraphics[scale=0.2]{Npic3.png}
\hfill

\phantom0

pic.3



The complete proof of this claim consists of consideration of 8 cases depending on which combinations of $\xi$ and $\delta_j$ surrounds the place of insertion as well as whether the two letters of the input involved in the two compositions come from the same input or from different ones.
\end{proof}

We illustrate he proof by the following example.

{\bf Example}

Consider the following two operations:

\vskip1cm

\phantom0\hskip3cm
\includegraphics[scale=0.05]{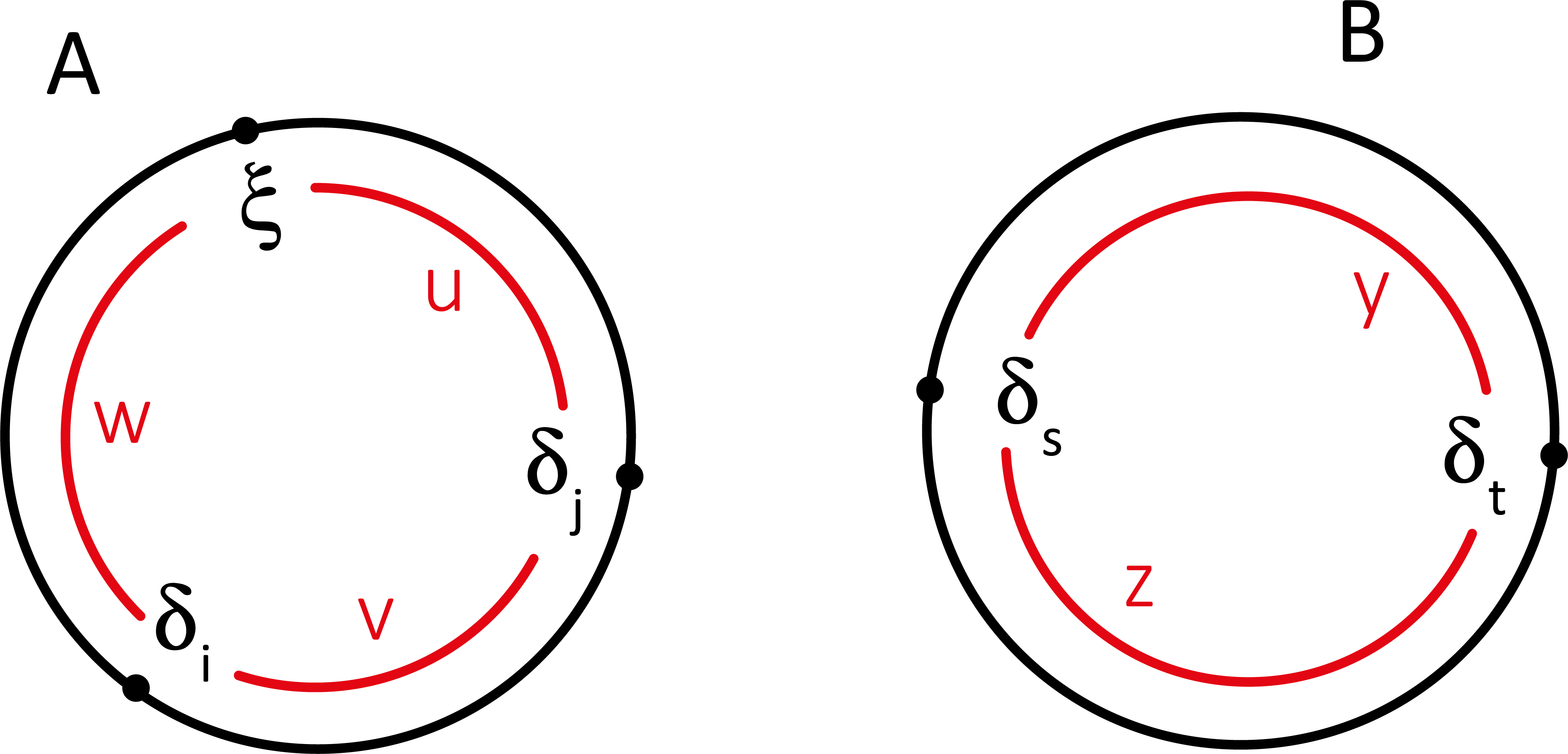}
\hfill

\phantom0

pic.4


Namely, let $A=\xi u\d_iv\d_jw\xi$ and $B=\d_3y\d_4z$. As an  input consider  three monomials $w_1$, $w_2$ and $w_3$ of the form
$$
w_1=ax_ib,\ \ \ w_2=cx_jdx_te,\ \ \ w_3=fx_sg.
$$

We fixed certain points in them, to construct coupling monomials which will cancel in $[U,W]$.
Consider first operation corresponding to a $\xi\delta$ monomial $A$. After inserting $x_i$ of $w_1$ and $x_j$ of $w_2$ into $\delta_i$ and $\delta_j$ of $A$, we get three outputs $aw$, $udx_te$ and $cvb$.
After inserting  $x_t$ from one of those  outputs $udx_te$, and $x_s$ of $w_3$  into the operation corresponding to $B$, we get four outputs for the composition $aw$, $cvb$, $udzg$ and $fye$. These are non-essential monomials since we started from letter $x_i$ in the
input $w_1$ of operation $U$, not from the internal letter in the $\xi\d$-presentation of operation $U$.

\vskip1cm

\phantom0\hskip4cm
\includegraphics[scale=0.25]{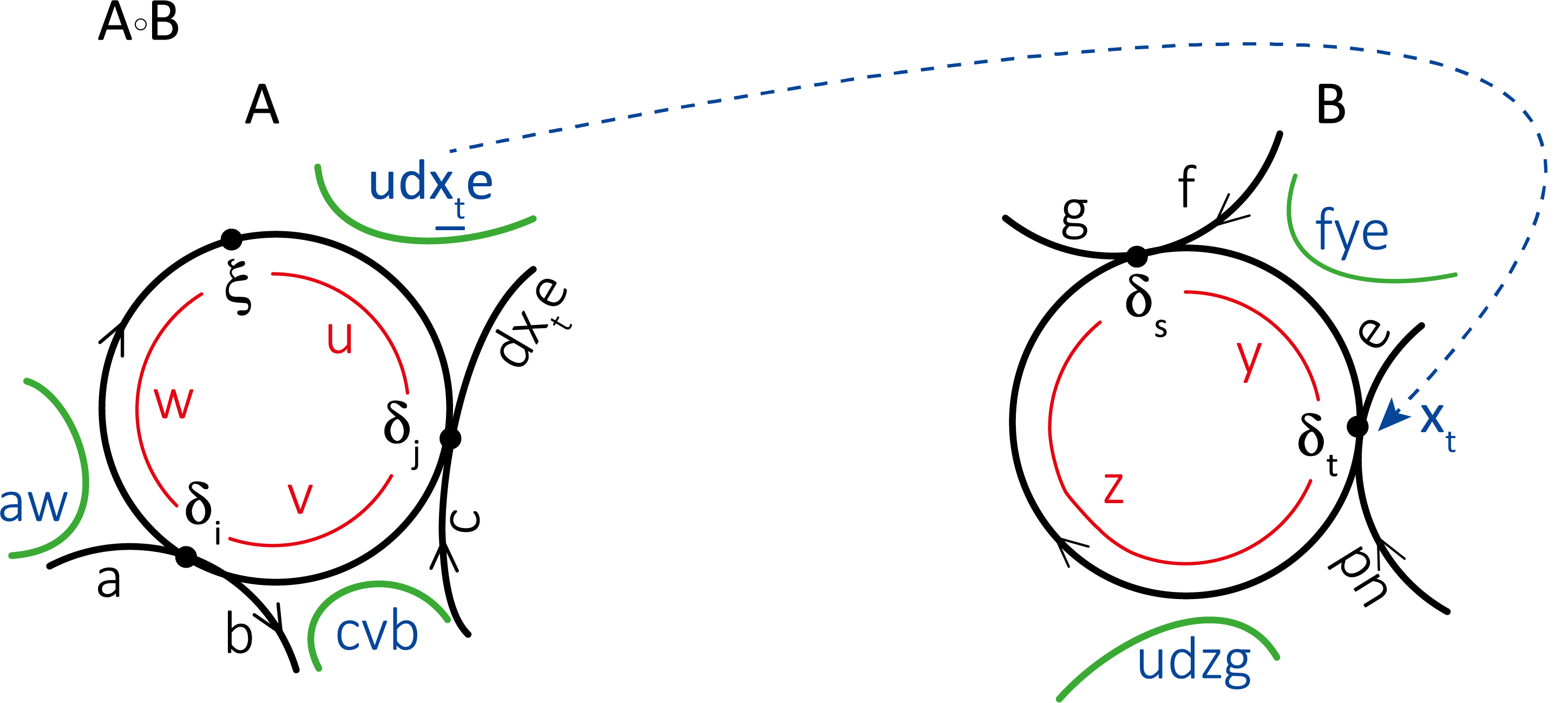}
\hfill

\phantom0

pic.5


Next, consider composition of operations $W\circ U$. If we insert $x_s$ of $w_3$ and $x_t$ of $w_2$ into the operation corresponding to $B$, we get two outputs $fye$ and $ucx_jdzg$. Now inserting $x_i$ of $w_1$, and $x_j$ of the second output into the operation defined by $A$, we get the same four outputs for the
composition: $aw$, $cvb$, $udzg$ and $fye$.

\vskip1cm

\phantom0\hskip4cm
\includegraphics[scale=0.06]{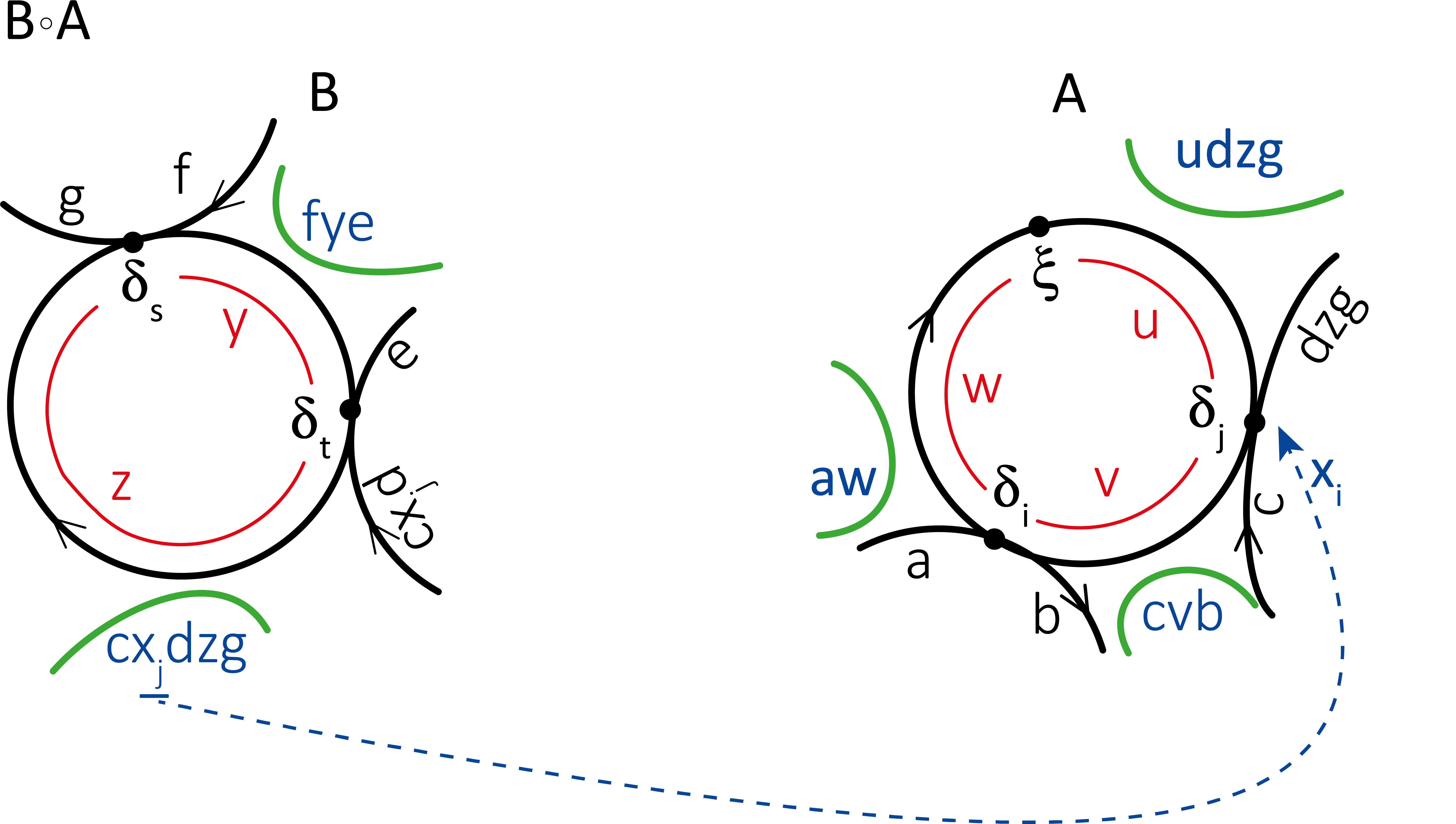}
\hfill

\phantom0

pic.6


 This example demonstrate that compositions in different order exhibit the same non-essential output monomials. As a matter of fact, the same argument works for every non-essential output monomial, so we can see that all non-essential outputs cancel in $[U,W]=U\circ W-W\circ U$.

\section{Homological purity and formality}\label{hp}

\subsection{Homological purity of the higher Hochschild complex}

The goal of this section will be to prove homological purity of the
complex $\zeta=\zeta^{(N)}_{\k}=\Hom_{\K}(R_{min}^{\oo N}, A_{cycl}^{\oo N})^{\Z_N}$ considered in previous sections.
It is a subcomplex of a complex $\tilde\zeta=\Hom_{\K}(R_{min}^{\oo N}, A_{cycl}^{\oo N}) $, which is the version
 where we do not take invariants under $\Z_N$-action, or in other words
 operations do have a fixed point. That is we have
$$
 \begin{array}{ccc}
 \zeta&\subset&{\tilde\zeta}\\
 \cap&&\cap\\
 \chH&\subset&\hH.\\
 \end{array}
 $$

 The underlying set of this complex  can be described as
$ \tilde\zeta=\{$monomials  $u\in \k\l\xi, x_i, \p_i \r$, $i=\bar{1,r},$ starting from $\xi$ or $\p_i \}_{\k}$. Whereas the underlying set of $\zeta$, consists of certain linear combinations of arbitrary monomials $u\in \K\l\xi,x_i,\delta_i\r$.

It will be instrumental in the proof of purity for $\zeta$.
Our proof, which is using Gr\"obner bases techniques in the ideals of free algebras, namely, for an ideal  generated  by one element associated to a differential, can be considered in a way as a construction of corresponding homotopy map.

We naturally have a bigrading on $\zeta=\oplus \zeta^k_m$: the grading  by $\p$-degree, and  by degree with respect to  $\xi$ and  $\p_i$th, $i=\bar{1,r}$, which we  call {\it weight}. That is, $u\in \zeta^k_m$, if $ \deg_{\xi, \p} u=:w(u)=m,$ and $ \deg_{\p} u=:g(u)=k$.
Essential for our considerations will be the  {\it cohomological grading} by $\xi$-degree:
$\zeta = \oplus \zeta(l)$, where $\zeta(l)=\mathop{\oplus}\limits_{m-k=l} \zeta^k_m.$

$$
\begin{array}{ccccccccc}
&&\delta&&\delta\delta&&\delta\delta\delta&&\delta\delta\delta\delta\\
&\nearrow&|&\nearrow&|&\nearrow&&\nearrow&\\
\xi&&\xi\delta&&\xi\delta\delta&&\xi\delta\delta\delta&&\\
&\nearrow&|&\nearrow&|&\nearrow&&&\\
\xi\xi&&\xi\xi\delta&&\xi\xi\delta\delta&&&&\\
&\nearrow&|&\nearrow&|&&&&\\
\xi\xi\xi&&\xi\xi\xi\delta&&&&&&\\
&\nearrow&|&&&&&&\\
\xi\xi\xi\xi&&&&&&&&
\end{array}
$$

\bigskip

If $\zeta_m$ is a subcomplex of $\zeta$, namely a slice consisting of elements of weight $m$, then we will use also splitting $\zeta_m=\mathop{\oplus}\limits_s \zeta_{m,s}$, where $s$ is an $x$-degree: $u\in \zeta_{m,s},$ if $w(u)=m,$ and $\deg_{x_1,...,x_r}(u)=s.$

The differential inherited from the complex $\Hom_{A^{\oo N}-mod-A^{\oo N}}(\oo_{i=1}^N A^{r_i}, A^{\oo N}_{cycl}), $
explained in  section~\ref{sH}, in terms of $\xi\d$-monomials boils down to the  following differential on $\tilde\zeta_{\K}$ (and on $\zeta_{\K}$):

$$
d(u_1\xi u_2\xi\dots u_n)=\sum(-1)^{g(u_1\xi u_2 \xi ... u_i)} u_1\xi u_2\xi\dots u_i\Delta u_{i+1}\dots u_n
$$
if $u_1\neq\varnothing$ ($u_1$ starting with $\delta_i$), here $\Delta=\sum\limits_{i=1}^r\delta_i x_i-x_i\delta_i$,
and
$$
d(\xi u_1\xi u_2\xi\dots u_n)=\xi d(u_1\xi u_2\xi\dots u_n)+\sum_{i=1}^r [\delta_i x_i u_1\xi u_2\xi\dots u_n-\delta_i u_1\xi u_2\xi\dots u_n x_i]
$$
if the monomial starts with $\xi.$ Here $u_i\in \K\l x_i, \delta_i\r.$

\begin{theorem}\label{main}
The homology of the complex $\zeta(A)=(\zeta, d)$, $\zeta=\oplus \zeta^k_m$, for $A=\k\l x_1,...,x_r\r, r\geq 2$ is sitting in the diagonal $k=m$,
 consequently, the complex $\zeta=\oplus \zeta(l)$, $\zeta(l)=\mathop{\oplus}\limits_{m-k=l} \zeta^k_m$  is pure, that is the homology is sitting only in the last place of this complex
 with respect to cohomological grading by $\xi$-degree. Thus the $\chH$ complex is pure as well.
\end{theorem}

\begin{proof} To prove this, we consider related complex $\widehat{\zeta}$
with the following differential:
$$
d_{\widehat\zeta}(u_1\xi u_2\xi\dots u_n)=\sum(-1)^{g(u_1\xi u_2 \xi ... u_i)} u_1\xi u_2\xi\dots u_i\Delta u_{i+1}\dots u_n,
$$
where $\Delta=\sum\limits_{i=1}^r\delta_ix_i-x_i\delta_i$.

We first prove that homologies are sitting in one place in the complex $(\widehat{\zeta},d_{\widehat\zeta})$ and then reduce the considerations for $(\tilde\zeta,d)$ to this. After that argument of Lemma~\ref{unbar} shows that for the subcomplex
$( \zeta,d) \subset (\tilde\zeta,d)$ the homology is also sitting in one place, if it is the case for $(\tilde\zeta,d)$. Since $\zeta$ is quasi-isomorphis to the $\chH$ complex $C^{(\bullet)}(A)$ for smooth algebras, we get that $C^{(\bullet)}(A)$ is also pure.

\begin{theorem}\label{D} The $m$-th slice of the complex $(\widehat{\zeta},d_{\widehat\zeta})$ over $A=\K\l x_1,...,x_r\r, r\geq 2,$
$$
\widehat{\zeta}_m=\{u\in {\mathbb K}\langle \xi, x_i,\delta_i\rangle: w(u)={\rm deg}_{\xi,\delta}u=m\}
$$
has non-trivial homology only in the last place with respect to cohomological grading by $\xi$-degree.

\end{theorem}

\begin{proof} Induction by $m$. For arbitrary $m$ we will first need to consider the
case of cohomological degree one, that is
the case of one $\xi$.

\begin{lemma}\label{D1xi}

Consider the place in $(\widehat{\zeta}_m, d_{\widehat\zeta})$,
where ${\rm deg}_\xi u=1$, $u\in \widehat{\zeta}_m$ (one but last place in the complexes $\widehat{\zeta}_m$). Then the homology in this place is trivial.
\end{lemma}

\begin{proof} Let $d_{\widehat\zeta}(u)=0$ for $u\in \widehat{\zeta}$ with ${\rm deg}_\xi u=1$. We show that $u\in{\rm Im}\,d_{\widehat\zeta}$. Since ${\rm deg}_\xi u=1$, $u$ has the shape
$$
u=\sum a_i\xi b_i,\ \ a_i,b_i\in {\mathbb K}\langle x_1\dots,x_r,\delta_1\dots\delta_r\rangle.
$$
Then
$$
d_{\widehat\zeta}u=\sum (-1)^{g(a_j)} a_j\Delta b_j=0.
$$

Consider the ideal $I$ in ${\mathbb K}\langle x_1\dots,x_r,\delta_1\dots\delta_r\rangle$ generated by $\Delta$: $I={\rm Id}(\Delta)$.

We will use notions of Gr\"obner bases theory  and the following  lemma (version of the Diamond  lemma) to describe when this above equality might happen. We remind them here.

\begin{definition}
Monomials $ u,v \in {\mathbb K}\langle Y \rangle$ form an ambiguity $(u,v),$ if for some $w\in {\mathbb K}\langle Y\rangle$, $uw=wv$.

\end{definition}

Suppose in ${\mathbb K}\langle Y \rangle$ we have fixed some admissible well-ordering (that is, ordering compatible with multiplication, which satisfies d.c.c.), for example, (left-to-right) degree-lexicographical ordering: we fix an order on variables, say $y_1 <...<y_n$, and compare monomials on $Y$ lexicographically (from left to right). Polynomials are compared by their highest terms.

\begin{definition}
 Let $u,v$ be  two monomials $u,v$, which are highest terms of the elements $U,V$ from the ideal $I\in {\mathbb K}\langle Y \rangle: U=u+\tilde u, V=v+\tilde v$, where $\tilde u, \tilde v \in {\mathbb K}\langle Y \rangle$, smaller than $u,v\in \l Y \r $ respectively: $\tilde u < u, \tilde v < v$. Then the resolution of the ambiguity $(u,v)$ formed by monomials $u,v$ is a polynomial $Uw-wV=\tilde u w-w\tilde v$, which is reducible to zero modulo  generators of an ideal.
\end{definition}

\begin{definition} A reduction on $\K \l Y \r$ modulo generators of an ideal $f_i=\bar f_i+\tilde  f_i$, where $\bar f_i$ is a highest term of $f_i$, is a collection of linear maps defined on monomials as follows:
$r_{u \bar f_i v}(w)= u \tilde f_i v$, if $w=u \bar f_i v$, and $w$ otherwise.
\end{definition}

The polynomial is called reducible to zero if there exists a sequence of reductions modulo generators of an ideal, which results in zero.

\begin{lemma}\label{Diam} (Version of Diamond Lemma~\cite{Mora}) Let
$A={\mathbb K}\langle y_1\dots,y_n\rangle/{\rm Id}(r_1,\dots,r_m)$. Let $M$ be the syzygy module of the relations $r_1,\dots,r_m$, that is $M$ is the submodule of the free ${\mathbb K}\langle y_1\dots,y_n\rangle$-bimodule generated by the symbols $\widehat{r_1},\dots \widehat{r_m}$ consisting of $\sum f_i\widehat{r_{s_i}}g_i$ such that $\sum f_ir_{s_i}g_i=0$.

Then $M$ is generated by trivial syzygies $\widehat{r_i}ur_j-r_iu\widehat{r_j}$ and the syzygies obtained by resolutions of ambiguities between highest terms of relations (with respect to some ordering).

\end{lemma}

Let us fix the ordering $\delta_1>\delta_2>\dots>x_1>x_2>\dots$. Then the leading term of the polynomial $\Delta$ is $\delta_1x_1$. It does not produce any ambiguities. Hence by Lemma~\ref{Diam} (version of Diamond Lemma), the corresponding syzygy module $M$ is generated by trivial syzygies, and therefore

$$
(*)\quad\quad  \sum  a_j\widehat\Delta b_j=\sum u_k(\widehat{\Delta}v_k\Delta-\Delta v_k\widehat{\Delta})w_k
$$

After we know this we can construct an element

$$
g= \sum {\gamma_k} u_k\xi v_k\xi w_k
$$

where $\gamma_k\in {\mathbb C}$ are chosen in such a way that in the following sum all summands have positive signs

$$
d_{\widehat\zeta}(g)=\sum  (u_k\xi v_k\Delta w_k-u_k \Delta v_k\xi w_k)
$$

We can see then that the
 latter expression is  just the same as the above formula $(*)$ with $\widehat\Delta$ substituted by  $\xi$, hence

$$
\sum  (u_k\xi v_k\Delta w_k-u_k \Delta v_k\xi w_k)=\sum  a_j\xi b_j=u.
$$

And we finally have

$$d_{\widehat\zeta}(g)=\sum  a_j\xi b_j=u.$$

\end{proof}

To continue the proof of Theorem~\ref{D}, we need a basis of induction. So we prove that in the complex $(\widehat{\zeta}_2,d_{\widehat\zeta})$
$$
0\to \dots\xi\dots\xi\dots \to \dots\xi\dots\delta\dots\to \dots\delta\dots\delta\dots\to 0
$$
the homology is sitting in the last place.

Since we already have Lemma~\ref{D1xi}, which deals with the case of one $\xi$ it is equivalent to proving exactness only in one
term, where $\xi$-degree is equal to two. That is, we need to show that if ${\rm deg}_\xi u=2$, ${\rm deg}_{\delta_i}u=0$, and $d_{\widehat\zeta}(u)=0$, then $u=0$. Write $u=\xi u_0+v$, where $v$ does not have $\xi$ on the first position. Then we have
$$
0=d_{\widehat\zeta}(u)=\xi d_{\widehat\zeta}(u_0)+\Delta u_0+d_{\widehat\zeta}(v).
$$

The only term starting with $\xi$ is $\xi d_{\widehat\zeta}(u_0)$, so $d_{\widehat\zeta}(u_0)=0$. Since ${\rm deg}_\xi u_0=1$, we are in situation of Lemma~\ref{D1xi} and $u_0=d_{\widehat\zeta}(v_0)$. Since $u_0$ is free from $\delta _i$, we have $u_0=0$. Thus $u=v=\sum x_i u_i$ and $0=d_{\widehat\zeta}(u)=\sum x_id_{\widehat\zeta}(u_i)$ implies $d_{\widehat\zeta}(u_i)=0$. Applying the same argument to $u_i$ repeatedly, we arrive at $u=0$, as required.

Step of induction in the proof of Theorem~\ref{D}. Let ${\rm deg}_{\xi,\delta_i}u=m$ and $u$ is homogeneous with respect to $\xi$ as well as with respect to $x_i,\xi,\delta_i$ and $u$ is not in the last term of the complex: ${\rm deg}_\xi u\geq 1$. We need to show that $u\in {\rm Im}(d_{\widehat\zeta})$ provided $d_{\widehat\zeta}(u)=0$. We present $u=\xi u_0+v$, where $v$ is not starting from $\xi$. Then
$$
d_{\widehat\zeta}(u)=\Delta u_0+\xi d_{\widehat\zeta}(u_0)+d_{\widehat\zeta}(v)=0.
$$

The only term starting with $\xi$ can not cancel with anything, so $d_{\widehat\zeta}(u_0)=0$. Now ${\rm deg}_{\xi,\delta_i}(u_0)=m-1$. By induction hypothesis and Lemma~\ref{D1xi} (if $m=2$), $u_0=d_{\widehat\zeta}(w)$. Consider $d_{\widehat\zeta}(\xi w)=\Delta w+\xi d_{\widehat\zeta}(w)$. Then
$$
u'=u-d_{\widehat\zeta}(\xi w)=\xi u_0+v-\Delta w-\xi d_{\widehat\zeta}(w)=v-\Delta w.
$$

Thus $u'$ equals $u$ modulo ${\rm Im}\,d_{\widehat\zeta}$ and does not have $\xi$ in the first position:
$$
u'=\sum x_j\xi u_j+\sum \delta_j\xi v_j+v,
$$
where $\xi$ is absent from $v$ in the first two positions. Then
$$
0=d_{\widehat\zeta}(u')=\sum x_j\Delta u_j-\sum \delta_j\Delta v_j+\sum x_j\xi d_{\widehat\zeta}(u_j)+\sum \delta_j\xi d_{\widehat\zeta}(v_j)+d_{\widehat\zeta}(v).
$$

Considering terms with $\xi$ in the second position, we deduce $d_{\widehat\zeta}(u_j)=d_{\widehat\zeta}(v_j)=0$ for all $j$. By the induction hypothesis $u_j=d_{\widehat\zeta}(w_j)$ and $v_j=d_{\widehat\zeta}(s_j)$. Now $u''$ equals $u'$ modulo ${\rm Im d_{\widehat\zeta}}$ and $u''$ has no $\xi$ in the first two positions, where
$$
u''=u-d_{\widehat\zeta}(\sum x_j\xi w_j+\sum \delta_j \xi s_j).
$$

After repeating this procedure, at the end we get $u=t\xi^m$ modulo ${\rm Im}\,d_{\widehat\zeta}$, where $\deg_\xi t=0$. Now $0=d_{\widehat\zeta}(t\xi^m)=td_{\widehat\zeta}(\xi^m)$ and therefore $t=0$ since $d_{\widehat\zeta}(\xi^m)=\Delta\xi^{m-1}+\dots+\xi^{m-1}\Delta\neq 0$. Hence $u\in {\rm Im}\,d_{\widehat\zeta}$.

\end{proof}

Now we prove the  theorem for $(\tilde\zeta,d)$.

\begin{theorem}\label{bar} The $m$-th slice of the complex $(\bar{\zeta},d)$,
$$
\bar{\zeta}_m=\{u\in {\mathbb K}\langle \xi, x_i,\delta_i\rangle: w(u)={\rm deg}_{\xi,\delta}u=m\}
$$
for each $m\geq 2$ has non-trivial homology only in the last place with respect to cohomological grading by $\xi$-degree.

\end{theorem}

\begin{proof}
First we need preliminary exactness result for the case of one $\xi$.

\begin{lemma}\label{d1xi}

Consider the place in $(\bar{\zeta}_m,d)$ for any $m\geq 2$, where ${\rm deg}_\xi u=1$, $u\in \bar{\zeta}_m$ (one but last place in the complex). Then the homology in this place is trivial.

\end{lemma}

\begin{proof} Let $u\in \bar{\zeta}$ be such that ${\rm deg}_\xi u=1$, ${\rm deg}_{\xi,\delta_i} u\geq 2$ and $d_{\tilde\zeta}(u)=0$. We have to show that $u\in {\rm Im}(d)$. Write
$$
u=\xi u_0+\sum a_i\xi b_i, \ \ \text{$a_i\neq$const}.
$$
Then
$$
0=d(u)=\sum \delta_i[x_i,u_0]+\sum(-1)^{\sigma}a_i\Delta b_i.
$$

Thus the following equality holds in $B= {\mathbb K}\langle x_1\dots,x_r,\delta_1\dots\delta_r\rangle/{\rm Id}(\Delta)$:

\begin{equation}\label{star}
0=d(u)=\sum \delta_i[x_i,u_0].
\end{equation}

\begin{lemma}\label{cen2} The equality $\sum \delta_i[x_i,u]=0$ in $B$ implies $[x_i,u]=0$ in $B$ for any $i$.
\end{lemma}

\begin{proof} Let us consider ordering $x_1>x_2>...>\delta_1>\delta_2>...$, then $\Delta$ forms a Gr\"obner basis. Take a normal form $N([x_i,u])
\in {\mathbb K}\langle x_1\dots,x_r,\delta_1\dots\delta_r\rangle=\K\l XD \r$ with respect to the Gr\"obner basis of the ideal ${\rm Id}(\Delta)$. In other words, we present the element $[x_i,u] \in  {\mathbb K}\langle x_1\dots,x_r,\delta_1\dots\delta_r\rangle$ as a sum of monomials which does not contain $x_1 \delta_1$ as a submonomial. Then element $N(\sum \p_i [x_i,u])= \sum \p_i N([x_i,u])=0$ in $\K\l XD \r$, hence
 $N[x_i,u]=0$ in $\K\l XD \r$, which means
 $[x_i,u]=0$ in $B$.
\end{proof}

\begin{lemma}\label{cen}(Centralizer) If in $B=
{\mathbb K}\langle x_1\dots,x_r,\delta_1\dots\delta_r\rangle/{\rm Id}(\Delta)$, $r\geq 2$, $[u,x_i]=0$ for all $i$, then $u\in{\mathbb K}$.
\end{lemma}

\begin{proof} Fix the ordering $\delta_1>\delta_2>\dots>x_1>x_2>\dots$ The highest term of $\Delta$ is $\delta_1x_1$. Then the set $N$ of corresponding normal words (those which do not contain $\delta_1x_1$) is closed under multiplication by $x_2$ on either side: $x_2N\subset N$ and $Nx_2\subset N$.

Let $u\in A$ and $[u,x_i]=0$ for all $i$. As every element of $A$, $u$ can be written as a linear combination of normal words: $u=\sum c_j w_j$, where $w_j\in N$ are pairwise distinct and $c_j$ are non-zero constants. Then $0=[u,x_2]=\sum c_j(w_jx_2-x_2w_j)$. Since $w_jx_2,x_2w_j\in N$, the last equality holds if and only if it holds in the free algebra. Hence $\sum c_j w_j$ commutes with $x_2$ in the free algebra and therefore $u\in {\mathbb K}[x_2]$. The same holds for any other $x_j$, $j\neq 1$ (they enter the game symmetrically) and therefore $u$ is in the intersection of ${\mathbb K}[x_j]$ as subalgebras of $A$. Since this intersection is ${\mathbb K}$, $u\in{\mathbb K}$. Note, that it is essential here that we have at least two $x$th: $r\geq 2$.
\end{proof}

From (\ref{star}), $[x_j,u_0]=0$ in $B$ for all $i$, according to Lemma~\ref{cen2}. By the centralizer lemma $u_0$ is a constant in $B$. Since $m\geq 2$, $u_0=0$ in $B$. Hence
$$
u_0=\sum s_i\Delta t_i
$$
in the free algebra. Thus
$$
u=\sum \xi s_i\Delta t_i+\sum a_i\xi b_i.
$$

Now we substitute $u$ with $u'=u({\rm mod} \, {\rm Im (d)})$, where
$$
u'=u-d(\sum (-1)^{{\rm deg}_{\delta}s_i}\xi s_i\xi t_i).
$$
After cancelations, we get
$$
u'=\sum a_i\xi b_i-\sum_{i,j} (-1)^{\sigma} \delta _j[x_j,s_i\xi t_i]
$$
and therefore $u'$ has no terms starting with $\xi$. Thus
$$
u'=\sum \delta_iu_i
$$
and we fall into the situation of the differential $d_{\widehat\zeta}$ on the complex $\hat \zeta$:
$$
d(u')=\sum \delta_id_{\widehat\zeta}(u_i)\iff d_{\widehat\zeta}(u_i)=0.
$$
By Theorem~\ref{D}, $u_i=d_{\widehat\zeta}(w_i)$ and
$$
u'=\sum \delta_id_{\widehat\zeta}(w_i)=d(-\sum\delta_iw_i),
$$
which yields that $u'$ and therefore $u$ belongs to ${\rm Im}\,d$.
\end{proof}

Now let $\deg_{\xi}u \geq 2.$ Suppose $du=0.$ We will show that $u \in {\rm Im}\, d.$ As before present it as $u= \xi u_0+ v,$ where $v$ does not start with $\xi$. Then
$0=du=\xi d_{\widehat\zeta} u_0+v'$, where $v'$ does not start with $\xi$, hence $d_{\widehat\zeta}u_0=0$. By Theorem~\ref{D} $u_0=d_{\widehat\zeta}s $ for some $s$. Thus take $u'=u-d(\xi s)=u-\xi d_{\widehat\zeta}s-...=\xi d_{\widehat\zeta}s+v-\xi d_{\widehat\zeta}s-...$, and we have a presentation of $u$ modulo the ideal ${\rm Im}\, d$ as an element with no $\xi$ at the first position: $u=\sum \d_i u_i$. Thus, $du=d_{\widehat\zeta}u$ and we can use Theorem~\ref{D} to ensure that $u \in {\rm Im} d$. Indeed, since $0=du=d_{\widehat\zeta}u$ and
$d_{\widehat\zeta}u=-\sum \d_id_{\widehat\zeta} u_i, \quad d_{\widehat\zeta}u_i=0$ for all $i$. Since $\deg_{\xi} u_i \geq 1$, by Theorem~\ref{D} we have $u_i=d_{\widehat\zeta}(w_i)$ for some $w_i$. Thus $u=\sum \d_i u_i
=-\sum \d_i d_{\widehat\zeta}w_i=d(\sum \d_iw_i)$. The latter equality $d(\sum \d_iw_i)=-\sum \d_i d_{\widehat\zeta}w_i$ holds because $\d_iw_i$ not starting with $\xi$. So $u\in {\rm Im} d$, and
 this completes the proof of Theorem~7.8.

\end{proof}

\begin{lemma}\label{unbar}
If the complex $(\tilde\zeta,d)$
is homologically pure, and homology is sitting in the last place w.r.t cohomological grading by $\xi$-degree, the same is true for the complex $(\zeta, d)$.
\end{lemma}

\begin{proof} The statement about $\zeta$ follows from the fact the  cyclization of the complex commutes with the differential in our case. This in turn deduced from the fact that the differential, given  precisely at the beginning of this section obviously commute with $\Z_N$ action.
\end{proof}

This lemma together with Theorem~\ref{bar} completes the proof of Theorem~7.1.

\end{proof}

\subsection{Homological purity of the higher Hochschild complex: quiver path algebra case}\label{q}

We consider here the case when $A$ is the path algebra of the Kronecker quiver, in stead of the case of free algebra. It allows to construct the notions of noncommutative projective geometry, as explained in \cite{KR}. Then we analogously proceed with the arbitrary quiver with at least two vertices.

The Gr\"obner bases theory, which in particular relates the generators of the syzygy module with ambiguities on the relations, as in lemma~\ref{Diam} works not only for ideals in free associative algebras, but in bigger generality of algebras with multiplicative basis \cite{EdG}. This includes path algebras of quivers. Namely, the algebra $A$ should have a linear basis $\Bc$, which is multiplicative in a sense that for any $m_1, m_2 \in \Bc, m_1\cdot_A m_2 \in \Bc $ or $m_1\cdot_A m_2=0$. Moreover, most essential (and not always present) condition is that there should exist a well-ordering on $\Bc$ (ordering satisfying d.c.c.), which is compartible with multiplication, that is for all $m_1,m_2,c \in \Bc, m_1 \leq m_2 $ implies $m_1 \cdot_A c \leq m_2 \cdot_A c$ and  $c \cdot_A m_1 \leq c \cdot_A m_2.$ This ordering is extended to the whole algebra from the basis $\Bc$ by saying that the element with the higher 'tip' (the highest element of $\Bc$, appearing with nonzero coefficient in the linear decomposition of an element from the algebra) is higher.

Path algebra of a quiver $\Qc$ is generated by  $\Qc_0$ - the set of vertices of $\Qc$ and $\Qc_1$ - the set of arrows of $\Qc$ subject to obvious relations. It has a multiplicative linear basis $\Bc$ consisting of all pathes along the quiver. The product of two pathes
$p,q \in\Bc, \, p=v_1a_1...a_{n-1}v_n, q=u_1b_1...b_{k-1}u_k$  is $p\cdot_Aq= pq$, if $v_n=u_1$, and zero otherwise. The admissible well-ordering on $\Bc$ can be as in the free associative algebra, for example, deg-lex ordering, when some ordering on generators  $\Qc_0 \cup \Qc_1$ is fixed.

Lemma~\ref{Diam} in the setting of path algebras  sounds as follows.

\begin{lemma}\label{Diam2} (Version of Diamond Lemma~\cite{EdG}) Let
$A=P\Qc/{\rm Id}(r_1,\dots,r_m)$, where $P\Qc$ is the path algebra of the quiver $\Qc$. Let $M$ be the syzygy module of the relations $r_1,\dots,r_m$, that is $M$ is the submodule of the free $\Qc$-bimodule generated by the symbols $\widehat{r_1},\dots \widehat{r_m}$ consisting of $\sum f_i\widehat{r_{s_i}}g_i$ such that $\sum f_ir_{s_i}g_i=0$.
\end{lemma}

Now the proof of Theorem~\ref{main} will work for the case of path algebra of the Kronecker quiver even with some simplifications. comparing to the case of free algebras. For, example, the centraliser lemma~\ref{cen} will get simplified, the requirement that the number of arrows is bigger than two can be dropped in this case.

We give now this version of the proof here.
Let $\Kc_r$ be the Kronecker quiver, namely the quiver with two vertices ($y=x_0$, $z=x_{r+1}$) and $r$ arrows  $x_1,...,x_r$ form $y$ to $z$.  Denote by $P\Kc_r$ the path algebra of this quiver with generators $x_0=y,x_1,...,x_r,x_{r+1}=z$.
Consider the $\chH$ complex $C^{(\bullet)}(P\Kc_r)$, it will have a small quasi-isomorphic (since $P\Kc_r$ is smooth) subcomplex $\zeta(P\Kc_r)$, defined in section~\ref{sH}.
At some point in section~\ref{sH}, when we depicted the basis  of $\zeta(A)$ (pic.1), we started to work with free associative algebra $A$. The new description of the $\xi\p$ basis in the case of $P\Kc_r$ one can obtain by saying that in pic.1 the monomials $u_i$ are monomials on $\p_i, \p'_i, i=0,r+1$ and the linear basis of $P\Kc_r$ $\Bc_{\Kc}=\{yx_i, yx_iz, x_iz \}$ (or from the linear basis of a path algebra $P\Qc$, in case of an arbitrary quiver $\Qc$).
In other words, $u_i$ are monomials from the free product
 $P\Kc_r * \K\l\p_0,...,\p_{r+1},\p'_0,...,\p'_{r+1}\r.$
The differential in $\zeta(P\Kc_r)$ in  terms of $\xi\d$-monomials over $A=P\Kc_r$ will be:
$$
d(u_1\xi u_2\xi\dots u_n)=\sum(-1)^{g(u_1\xi u_2 \xi ... u_i)} u_1\xi u_2\xi\dots u_i\Delta' u_{i+1}\dots u_n
$$
if $u_1\neq\varnothing$ ($u_1$ starting with $\delta_i$), here $\Delta'=\sum\limits_{i=0}^{r+1}\delta_i x_i-x_i\delta_i$,
and
$$
d(\xi u_1\xi u_2\xi\dots u_n)=\xi d(u_1\xi u_2\xi\dots u_n)+\sum_{i=0}^{r+1} [\delta_i x_i u_1\xi u_2\xi\dots u_n-\delta_i u_1\xi u_2\xi\dots u_n x_i]
$$
if the monomial starts with $\xi.$ Here  $u_i \in P\Kc_r * \K\l\p_0,...,\p_{r+1},\p'_0,...,\p'_{r+1}\r.$


\begin{theorem}\label{main2}
The homology of the complex $\zeta(A)=(\zeta, d)$, $\zeta=\oplus \zeta^k_m$, for $A=P\Kc_r$ is sitting in the diagonal $k=m$,
 consequently, the complex $\zeta=\oplus \zeta(l)$, $\zeta(l)=\mathop{\oplus}\limits_{m-k=l} \zeta^k_m$  is pure, that is the homology is sitting only in the last place of this complex
 with respect to cohomological grading by $\xi$-degree. Thus the $\chH$ complex over  $A=P\Kc_r$  is pure as well.
\end{theorem}

\begin{proof} To prove this, we consider related complex $\widehat{\zeta}$
with the following differential:
$$
d_{\widehat\zeta}(u_1\xi u_2\xi\dots u_n)=\sum(-1)^{g(u_1\xi u_2 \xi ... u_i)} u_1\xi u_2\xi\dots u_i\Delta' u_{i+1}\dots u_n,
$$
where $\Delta'=\sum\limits_{i=0}^{r+1}\delta_ix_i-x_i\delta_i$.

We first prove that homologies are sitting in one place in the complex $(\widehat{\zeta},d_{\widehat\zeta})$ and then reduce the considerations for $(\tilde\zeta,d)$ to this. After that argument of Lemma~\ref{unbar} shows that for the subcomplex
$( \zeta,d) \subset (\tilde\zeta,d)$ the homology is also sitting in one place, if it is the case for $(\tilde\zeta,d)$. Since $\zeta$ is quasi-isomorphis to the $\chH$ complex $C^{(\bullet)}(A)$ for smooth algebras, we get that $C^{(\bullet)}(A)$ is also pure.

\begin{theorem}\label{D2} The $m$-th slice of the complex $(\widehat{\zeta},d_{\widehat\zeta})$ over $A=P\Kc_r$
$$
\widehat{\zeta}_m=\{u\in {\mathbb K}\langle \xi, x_i,\delta_i\rangle: w(u)={\rm deg}_{\xi,\delta}u=m\}
$$
has non-trivial homology only in the last place with respect to cohomological grading by $\xi$-degree.

\end{theorem}

\begin{proof} Induction by $m$. For arbitrary $m$ we will first need to consider the
case of cohomological degree one, that is
the case of one $\xi$.

\begin{lemma}\label{D1xi2}

Consider the place in $(\widehat{\zeta}_m, d_{\widehat\zeta})$,
where ${\rm deg}_\xi u=1$, $u\in \widehat{\zeta}_m$ (one but last place in the complexes $\widehat{\zeta}_m$). Then the homology in this place is trivial.
\end{lemma}

\begin{proof} Let $d_{\widehat\zeta}(u)=0$ for $u\in \widehat{\zeta}$ with ${\rm deg}_\xi u=1$. We show that $u\in{\rm Im}\,d_{\widehat\zeta}$. Since ${\rm deg}_\xi u=1$, $u$ has the shape
$$
u=\sum a_i\xi b_i,\ \ a_i,b_i\in E=P\Kc_r *\K\l \delta_0,...,\delta_{r+1},\p'_0,...,\p'_{r+1}\rangle.
$$
Then
$$
d_{\widehat\zeta}u=\sum (-1)^{g(a_j)} a_j\Delta' b_j=0.
$$

Consider the ideal $I$ in $E=P\Kc_r *   \K\langle \delta_0,...,\delta_{r+1},\p'_0,...,\p'_{r+1}\rangle$ generated by $\Delta'$: $I={\rm Id}(\Delta')$.
To use the Gr\"obner bases theory for  $E$ we will present it as a quotient of the path algebra of another quiver $\Rc$. Namely, $E=P\Kc_r *   \K\langle \delta_0,...,\delta_{r+1},\p'_0,...,\p'_{r+1}\rangle=P\Rc/{\rm Id}(\p_i'=\p_i)$, where $\Rc$ is a quiver with two vertices $x_0=y, x_{r+1}=z$, $r$ arrows between them and a rose in each vertex,  consisting of
arrows $\p_0,...,\p_{r+1}$ in vertex $x_0$, and of arrows $\p'_0,...,\p'_{r+1}$ in  vertex $x_{r+1}$.
If we factor out relations identifying each $\p_i$ with $\p_i'$, we get exactly the free product  $P\Kc_r *   \K\langle \delta_0,...,\delta_{r+1},\p'_0,...,\p'_{r+1}\rangle$. Thus now we consider an ideal $I$ in the path algebra of a quiver $\Rc$, $F=P\Rc$ generated by two groups of relations: $I={\rm Id}(\p'_i=\p_i, \Delta')$, and use the Gr\"obner bases theory in $F$ to show when the above equality may happen.

Let us fix the ordering $\p'_0>\dots>\p'_{r+1}>\delta_0>\dots>\delta_{r+1}>\dots>x_0>x_1>\dots>x_{r+1}$. Then the leading term of the polynomial $\Delta'$ is $\delta_0x_0$. It does not produce any  ambiguities with itself, or with relations $\p'_i=\p_i$. The latter relations does not produce any ambiguities between themselves either. Hence by Lemma~\ref{Diam2} (version of Diamond Lemma), the corresponding syzygy module $M$ is generated by trivial syzygies, and therefore

$$
(*)\quad\quad  \sum  a_j\widehat\Delta b_j=\sum u_k(\widehat{\Delta}v_k\Delta'-\Delta' v_k\widehat{\Delta})w_k
$$

After we know this we can construct an element
$$
g= \sum {\gamma_k} u_k\xi v_k\xi w_k
$$
where $\gamma_k\in {\mathbb C}$ are chosen in such a way that in the following sum all summands have positive signs
$$
d_{\widehat\zeta}(g)=\sum  (u_k\xi v_k\Delta' w_k-u_k \Delta' v_k\xi w_k)
$$

We can see then that the
 latter expression is  just the same as the above formula $(*)$ with $\widehat\Delta$ substituted by  $\xi$, hence
$$
\sum  (u_k\xi v_k\Delta' w_k-u_k \Delta' v_k\xi w_k)=\sum  a_j\xi b_j=u.
$$

And we finally have
$$d_{\widehat\zeta}(g)=\sum  a_j\xi b_j=u.$$
\end{proof}

To continue the proof of Theorem~\ref{D2}, we need a basis of induction. So we prove that in the complex $(\widehat{\zeta}_2,d_{\widehat\zeta})$
$$
0\to \dots\xi\dots\xi\dots \to \dots\xi\dots\delta\dots\to \dots\delta\dots\delta\dots\to 0
$$
the homology is sitting in the last place.

Since we already have Lemma~\ref{D1xi2}, which deals with the case of one $\xi$ it is equivalent to proving exactness only in one
term, where $\xi$-degree is equal to two. That is, we need to show that if ${\rm deg}_\xi u=2$, ${\rm deg}_{\delta_i}u=0$, and $d_{\widehat\zeta}(u)=0$, then $u=0$. Write $u=\xi u_0+v$, where $v$ does not have $\xi$ on the first position. Then we have
$$
0=d_{\widehat\zeta}(u)=\xi d_{\widehat\zeta}(u_0)+\Delta' u_0+d_{\widehat\zeta}(v).
$$

The only term starting with $\xi$ is $\xi d_{\widehat\zeta}(u_0)$, so $d_{\widehat\zeta}(u_0)=0$. Since ${\rm deg}_\xi u_0=1$, we are in situation of Lemma~\ref{D1xi2} and $u_0=d_{\widehat\zeta}(v_0)$. Since $u_0$ is free from $\delta _i$, we have $u_0=0$. Thus $u=v=\sum x_i u_i$ and $0=d_{\widehat\zeta}(u)=\sum x_id_{\widehat\zeta}(u_i)$ implies $d_{\widehat\zeta}(u_i)=0$. Applying the same argument to $u_i$ repeatedly, we arrive at $u=0$, as required.

Step of induction in the proof of Theorem~\ref{D2}. Let ${\rm deg}_{\xi,\delta_i}u=m$ and $u$ is homogeneous with respect to $\xi$ as well as with respect to $x_i,\xi,\delta_i$ and $u$ is not in the last term of the complex: ${\rm deg}_\xi u\geq 1$. We need to show that $u\in {\rm Im}(d_{\widehat\zeta})$ provided $d_{\widehat\zeta}(u)=0$. We present $u=\xi u_0+v$, where $v$ is not starting from $\xi$. Then
$$
d_{\widehat\zeta}(u)=\Delta' u_0+\xi d_{\widehat\zeta}(u_0)+d_{\widehat\zeta}(v)=0.
$$

The only term starting with $\xi$ can not cancel with anything, so $d_{\widehat\zeta}(u_0)=0$. Now ${\rm deg}_{\xi,\delta_i}(u_0)=m-1$. By induction hypothesis and Lemma~\ref{D1xi2} (if $m=2$), $u_0=d_{\widehat\zeta}(w)$. Consider $d_{\widehat\zeta}(\xi w)=\Delta' w+\xi d_{\widehat\zeta}(w)$. Then
$$
u'=u-d_{\widehat\zeta}(\xi w)=\xi u_0+v-\Delta' w-\xi d_{\widehat\zeta}(w)=v-\Delta' w.
$$
Thus $u'$ equals $u$ modulo ${\rm Im}\,d_{\widehat\zeta}$ and does not have $\xi$ in the first position:
$$
u'=\sum x_j\xi u_j+\sum \delta_j\xi v_j+v,
$$
where $\xi$ is absent from $v$ in the first two positions. Then
$$
0=d_{\widehat\zeta}(u')=\sum x_j\Delta' u_j-\sum \delta_j\Delta' v_j+\sum x_j\xi d_{\widehat\zeta}(u_j)+\sum \delta_j\xi d_{\widehat\zeta}(v_j)+d_{\widehat\zeta}(v).
$$

Considering terms with $\xi$ in the second position, we deduce $d_{\widehat\zeta}(u_j)=d_{\widehat\zeta}(v_j)=0$ for all $j$. By the induction hypothesis $u_j=d_{\widehat\zeta}(w_j)$ and $v_j=d_{\widehat\zeta}(s_j)$. Now $u''$ equals $u'$ modulo ${\rm Im}\, d_{\widehat\zeta}$ and $u''$ has no $\xi$ in the first two positions, where
$$
u''=u-d_{\widehat\zeta}(\sum x_j\xi w_j+\sum \delta_j \xi s_j).
$$

After repeating this procedure, at the end we get $u=t\xi^m$ modulo ${\rm Im}\,d_{\widehat\zeta}$, where $\deg_\xi t=0$. Now $0=d_{\widehat\zeta}(t\xi^m)=td_{\widehat\zeta}(\xi^m)$ and therefore $t=0$ since $d_{\widehat\zeta}(\xi^m)=\Delta'\xi^{m-1}+\dots+\xi^{m-1}\Delta'\neq 0$. Hence $u\in {\rm Im}\,d_{\widehat\zeta}$.
\end{proof}

Now we prove the  theorem for $(\tilde\zeta,d)$.

\begin{theorem}\label{bar2} The $m$-th slice of the complex $(\bar{\zeta},d)$,
$$
\bar{\zeta}_m=\{u\in {\mathbb K}\langle \xi, x_i,\delta_i\rangle: w(u)={\rm deg}_{\xi,\delta}u=m\}
$$
for each $m\geq 2$ has non-trivial homology only in the last place with respect to cohomological grading by $\xi$-degree.

\end{theorem}

\begin{proof}
First we need preliminary exactness result for the case of one $\xi$.

\begin{lemma}\label{d1xi2}
Consider the place in $(\tilde{\zeta}_m,d)$ for any $m\geq 2$, where ${\rm deg}_\xi u=1$, $u\in \tilde{\zeta}_m$ (one but last place in the complex). Then the homology in this place is trivial.
\end{lemma}

\begin{proof} Let $u\in \tilde{\zeta}$ be such that ${\rm deg}_\xi u=1$, ${\rm deg}_{\xi,\delta_i} u\geq 2$ and $d_{\tilde\zeta}(u)=0$. We have to show that $u\in {\rm Im}(d)$. Write
$$
u=\xi u_0+\sum a_i\xi b_i, \ \ \text{$a_i\neq$const}.
$$
Then
$$
0=d(u)=\sum \delta_i[x_i,u_0]+\sum(-1)^{\sigma}a_i\Delta' b_i.
$$

Thus the following equality holds in $G=P\Rc/{\rm Id}(\Delta',\p'_i=\p_i)$
\begin{equation}\label{star2}
0=d(u)=\sum \delta_i[x_i,u_0].
\end{equation}

\begin{lemma}\label{cen2N} The equality $\sum \delta_i[x_i,u]=0$ in $G$ implies $[x_i,u]=0$ in $G$ for any $i$.
\end{lemma}

\begin{proof} Let us consider ordering $x_0>x_1>...>\p'_0>\p'_1>...>\delta_0>\delta_1>...$, then $\Delta'$ and relations $\p'_i=\p_i$ form a Gr\"obner basis. Take a normal form $N([x_i,u])
\in P\Rc$ with respect to the Gr\"obner basis of the ideal ${\rm Id}(\Delta', \p'_i=\p_i)$. In other words, we present the element $[x_i,u] \in P\Rc$
 as a sum of monomials which does not contain $x_0 \delta_0$ or $\p'_i$ as a submonomial.
 Note that multiplication by $\p_i$ from the left preserves normality of the word.
 The element $N(\sum \p_i [x_i,u])= \sum \p_i N([x_i,u])=0$ in $F=P\Rc$, hence
 $N[x_i,u]=0$ in $P\Rc$, which means
 $[x_i,u]=0$ in $G=P\Rc/{\rm Id}(\Delta', \p'_i=\p_i) $.
\end{proof}

\begin{lemma}\label{cenN}(Centralizer) If in
$G=P\Rc/{\rm Id}(\Delta', \p'_i=\p_i) $.
 $[u,x_i]=0$ for all $i$, then $u\in{\mathbb K}$.
\end{lemma}

\begin{proof}
Since in the Kronecker quiver we always have two vertices $x_0, x_{r+1}$, and they have a property that centraliser of each vertex in the path algebra is only multipes of this vertex itself, it is enough to notice that $[u,x_0]=0$ implies $u=\alpha x_0$, $[u,x_{r+1}]=0$ implies $u=\beta x_{r+1}$, and since $x_0\neq x_{r+1}$, we have that $u\in \K$.
\end{proof}

From (\ref{star2}), $[x_j,u_0]=0$ in $G$ for all $i$, according to Lemma~\ref{cen2N}. By the centralizer lemma $u_0$ is a constant in $G$. Since $m\geq 2$, $u_0=0$ in $G$. Hence
$$
u_0=\sum s_i\Delta' t_i
$$
in the free path algebra $P\Rc$. Thus
$$
u=\sum \xi s_i\Delta' t_i+\sum a_i\xi b_i.
$$

Now we substitute $u$ with $u'=u({\rm mod} \, {\rm Im (d)})$, where
$$
u'=u-d(\sum (-1)^{{\rm deg}_{\delta}s_i}\xi s_i\xi t_i).
$$
After cancelations, we get
$$
u'=\sum a_i\xi b_i-\sum_{i,j} (-1)^{\sigma} \delta _j[x_j,s_i\xi t_i]
$$
and therefore $u'$ has no terms starting with $\xi$. Thus
$$
u'=\sum \delta_iu_i
$$
and we fall into situation of the differential $d_{\widehat\zeta}$ on the complex $\hat \zeta$:
$$
d(u')=\sum \delta_id_{\widehat\zeta}(u_i)\iff d_{\widehat\zeta}(u_i)=0.
$$
By Theorem~\ref{D}, $u_i=d_{\widehat\zeta}(w_i)$ and
$$
u'=\sum \delta_id_{\widehat\zeta}(w_i)=d(-\sum\delta_iw_i),
$$
which yields that $u'$ and therefore $u$ belongs to ${\rm Im}\,d$.
\end{proof}

Now let $\deg_{\xi}u \geq 2.$ Suppose $du=0.$ We will show that $u \in {\rm Im}\, d.$ As before present it as $u= \xi u_0+ v,$ where $v$ does not start with $\xi$. Then
$0=du=\xi d_{\widehat\zeta} u_0+v'$, where $v'$ does not start with $\xi$, hence $d_{\widehat\zeta}u_0=0$. By Theorem~\ref{D} $u_0=d_{\widehat\zeta}s $ for some $s$. Thus take $u'=u-d(\xi s)=u-\xi d_{\widehat\zeta}s-...=\xi d_{\widehat\zeta}s+v-\xi d_{\widehat\zeta}s-...$, and we have a presentation of $u$ modulo the ideal ${\rm Im}\, d$ as an element with no $\xi$ at the first position: $u=\sum \d_i u_i$. Thus, $du=d_{\widehat\zeta}u$ and we can use Theorem~\ref{D2} to ensure that $u \in {\rm Im} d$. Indeed, since $0=du=d_{\widehat\zeta}u$ and
$d_{\widehat\zeta}u=-\sum \d_id_{\widehat\zeta} u_i, \quad d_{\widehat\zeta}u_i=0$ for all $i$. Since $\deg_{\xi} u_i \geq 1$, by Theorem~\ref{D2} we have $u_i=d_{\widehat\zeta}(w_i)$ for some $w_i$. Thus $u=\sum \d_i u_i
=-\sum \d_i d_{\widehat\zeta}w_i=d(\sum \d_iw_i)$. The latter equality $d(\sum \d_iw_i)=-\sum \d_i d_{\widehat\zeta}w_i$ holds because $\d_iw_i$ not starting with $\xi$. So $u\in {\rm Im} d$, and
 this completes the proof of Theorem~\ref{bar2}.

\end{proof}

The passing from the complex $\tilde\zeta$ to the complex $\zeta$ goes exactly as before.

\begin{lemma}\label{unbar}
If the complex $(\tilde\zeta,d)$
is homologically pure, and homology is sitting in the last place w.r.t cohomological grading by $\xi$-degree, the same is true for the complex $(\zeta, d)$.
\end{lemma}


This lemma together with Theorem~\ref{bar2} completes the proof of Theorem~\ref{main2}

\end{proof}

We considered here in details  the case of the Kronecker quiver, since this case produces the notions of noncommutative projective geometry, as it is explained in \cite{KR}. But the same argument works in the case of the path algebra of an arbitrary quiver as well. One should consider the initial quiver, where additionally in each vertex there is a rose quiver with $n$  petals, where $n$ is a number of generators (vertices and arrows) of the initial path algebra. Petals are labeled by $\p_i^j$, $i=1,n$, $j$ stands for the number of the vertex. Then we consider path algebra of this quiver, factorised by the ideal generated by the relations $\p_i^j=\p_i^k$ for all $j,k$. The resulting algebra is isomorphic to the free product of the $P\Qc * \K\l \p_i^j\r$, where $\Qc$ is an initial quiver, and we apply the Gr\"obner bases theory for this algebra exactly as before. The centralizer lemma for arbitrary quiver  holds as soon as quiver has at least two vertices, since then there exist a pair of vertices, for which the  intersection of their centralizers in the path algebra is $\K$.

The resulting fact is that the theorem~\ref{main2} holds for the path algebra of an arbitrary quiver with at least two vertices.

\begin{theorem}\label{main3}Let $A$ be a path algebra  of an arbitrary quiver $\Qc$  with at least two vertices, $A=P\Qc$.
The homology of the complex $\zeta(A)=(\zeta, d)$, $\zeta=\oplus \zeta^k_m$,  is sitting in the diagonal $k=m$,
 consequently, the complex $\zeta=\oplus \zeta(l)$, $\zeta(l)=\mathop{\oplus}\limits_{m-k=l} \zeta^k_m$  is pure, that is the homology is sitting only in the last place of this complex
 with respect to cohomological grading by $\xi$-degree. Thus the $\chH$ complex over  $A=P\Qc$  is pure as well.
\end{theorem}

Combinations of theorems \ref{main2} and \ref{main3} gives the result for path algebra $A=P\Qc$  of an arbitrary quiver $\Qc$, except for the quiver with one vertex and one loop.

\subsection{Formality}

The important consequence of the result on the homological purity of the higher cyclic Hochschild complex  is that we derive formality  for these  complexes in $L_\infty$ sense \cite{KP}.
Various aspects of formality have been studied extensively
 (for example \cite{D, BEGM, BM, ts, sh}),  some of them are famously difficult.
 The main point of \cite{KP} is that for the purposes of deformation theory the weaker property of $L_{\infty}$ formality is what is essential. It was shown there that not only two quasi-isomorphic DGLAs give rise to the isomorphic deformation functors, but there is much weaker equivalence relation on DGLAs, that of $L_{\infty}$-equivalence, which actually defines the deformation functor.

 \begin{definition}
The DGLA $(C,d)$ is called {\it formal} if it is quasi-isomorphic to its cohomologies
$H^{\bullet}C.$
\end{definition}

\begin{definition}
The DGLA $(C,d)$ is $L_\infty-formal$    if it is $L_\infty$ quasi-isomorphic to its cohomologies
$(H^{\bullet}C, 0)$, considered with zero differential, that is there exists an $L_\infty$-morphism, which is a quasi-isomorphism of complexes.
\end{definition}

\begin{theorem} The higher cyclic Hochschild complex $\chH(A)=C^{(\bullet)}(A)=\prod\limits_N C_{cycl}^{(N)} (A)$ is $L_{\infty}$-formal, for free algebra $A=\k\l x_1,...x_n\r$ with at least two generators or path algebra $P\Qc$ of an arbitrary quiver $\Qc$ with at least two vertices.
\end{theorem}

In other words, the theorem holds for  path algebra $P\Qc$ of an arbitrary quiver$\Qc$, except for the quiver with one vertex and one loop.

\begin{proof} Remind that in the higher cyclic Hochschild complex  $\chH=\mathop{\prod}\limits_N C_{cycl}^{(N)} (A)$  we have the following grading and the subcomplex  $\zeta$ quasi-isomorphic to this complex is situated with respect to this grading in the following way: $\chH=\mathop{\oplus}\limits_{i \in \Z} \chH{(i)} $, where $i$ is a number of outputs minus number of inputs of corresponding operation.
Thus the embedding of $\zeta$ into $\chH$, described in section~\ref{sH} place zero component of $\zeta$ into zero component of $\chH$:
 $\zeta(0) \subseteq \chH(0)$,  where $\zeta$ considered again with the cohomological grading by $\xi$-degree.
Hence our main Theorem~\ref{main} ensures that the homology of $\chH$ is sitting in the zero place of the grading.
Let us consider the group action on $\chH$ induced by scaling, namely, $\C^*$ acts by $\lambda(u)=\lambda^m u$ for $u\in \chH(m)$. This means that the action uniquely defines the grading.

Now consider $(H^{\bullet}\chH, \infty)$, the $L_{\infty}$-structure on the homologies obtained by the homotopy transfer
of Kadeishvili \cite{Kd}, constructive description of which is given in \cite{KTF}, one can find  explanations also in \cite{BV}.
Since we had a reductive group action on $(\chH,0)$ this action can be pulled through to
$(H^{\bullet}\chH, \infty)$ and so will be compatible with the new $L_{\infty}$-structure on
$H^{\bullet}\chH$ again. Thus the grading on $(H^{\bullet}\chH, \infty)$, being defined by this action, is also natural, i.e. only zero component of it will be nontrivial.

Obviously, if there is only one component in the grading of $L_{\infty}$-algebra, then only one multiplication from $L_{\infty}$-structure can be non-zero. Since we shown that homology
$(H^{\bullet}\chH, \infty)$ is sitting in zero component only, and we are using convention where binary multiplication in the infinity structure has degree zero, only multiplication $m_2$ can be present. Thus in the $L_\infty$-structure of $(H^{\bullet}\chH, \infty)$ $m_n=0$ for $ n\geq 3$, and this implies formality.
Indeed, for formality we need to show that $(\chH,d)$ is quasi-isomorphic to $(H^{\bullet}\chH, 0)$. Since for the $L_{\infty}$-structure obtained by the homotopy transfer it is always true that $(H^{\bullet}\chH, \infty)$  is qiso to $(\chH,d)$, it is enough to show that $(H^{\bullet}\chH, 0)$ is qiso to $(H^{\bullet}\chH, \infty)$,
and this is obviously the case when $m_n=0, n\geq 3.$

\end{proof}



\scshape

\noindent   Natalia Iyudu, School of Mathematics,  The University of Edinburgh, JCMB, The King's Buildings, Edinburgh, Scotland EH9 3FD

\noindent E-mail address: \quad { n.iyudu@ed.ac.uk, n.iyudu@ihes.fr }\ \ \


\scshape

\noindent Maxim Kontsevich, Institut des Hautes \'Etudes Scientifiques, 35 route de Chartres, F - 91440 Bures-sur-Yvette

\noindent  E-mail address: \quad { maxim@ihes.fr}\\

\vspace{13mm}

\scshape

\end{document}